%% file: main.tex
\documentclass[10pt]{amsart}
\include{preamble}
\begin{document}
\title{A nonlinear Variant of Ball's Inequality}
\author{Jennifer Duncan}
\date{18 May, 2022}
\maketitle

\begin{abstract}
We adapt an induction-on-scales argument of Bennett, Bez, Buschenhenke, Cowling, and Flock to establish a global near-monotonicity statement for the nonlinear Brascamp--Lieb functional under a certain heat-flow, from which follows a global stability result for nonlinear Brascamp--Lieb inequalities under bounded perturbations.
\end{abstract}
\section{Introduction}
\subsection{Linear Brascamp--Lieb Inequalities}\label{intro}
For each $j\in\{1,...,m\}$, let $L_j:\mathbb{R}^n\rightarrow\mathbb{R}^{n_j}$ be a linear surjection and $p_j\in[0,1]$. The \emph{Brascamp--Lieb inequality} associated with the pair $(\txtbf{L},\txtbf{p}):=((L_j)_{j=1}^m,(p_j)_{j=1}^m)$ is the following:
\begin{align}
\int_{\mathbb{R}^n}\prod_{j=1}^m(f_j\circ L_j)^{p_j}\leq C\prod_{j=1}^m\left(\int_{\mathbb{R}^{n_j}}f_j\right)^{p_j}\hspace{5pt}\forall f_j\in L^1(\mathbb{R}^{n_j}), f_j\geq 0.
\label{eq:Bras1}
\end{align}
Using the notation of \cite{bcct}, we refer to the pair $(\txtbf{L},\txtbf{p})$ as a \emph{Brascamp--Lieb datum}. We define the \emph{Brascamp--Lieb constant}, BL(\txtbf{L},\txtbf{p}), to be the infimum over all constants $C\in(0,\infty]$ for which the above inequality holds. We define the \emph{Brascamp--Lieb functional} as 
\begin{align}
    \BL(\txtbf{L},\txtbf{p};\txtbf{f}):=\frac{\int_{\mathbb{R}^n}\prod_{j=1}^m(f_j\circ L_j)^{p_j}}{\prod_{j=1}^m\left(\int_{\mathbb{R}^{n_j}}f_j\right)^{p_j}},
\end{align}
for all $m$-tuples of non-zero non-negative functions $\txtbf{f}=(f_j)_{j=1}^m\in\prod_{j=1}^mL^1(\R^{n_j})$. We may hence write the Brascamp--Lieb constant as $\BL(\txtbf{L},\txtbf{p})=\sup_{\txtbf{f}}\BL(\txtbf{L},\txtbf{p};\txtbf{f})$.\par
The Brascamp--Lieb inequalities are a natural generalisation of many classical multilinear inequalities that commonly arise in analysis, examples of which include H\"older's inequality, Young's convolution inequality, and the Loomis--Whitney inequality. They have had a significant impact on a broad range of areas of mathematics and the sciences---it was developments in the study of Brascamp--Lieb inequalities that lead to the resolution of the century-old Vinagradov mean value conjecture \cite{bourgain2016proof}, which is now a celebrated theorem in analytic number theory. Other deep number-theoretic connections were established by Christ et al. \cite{christ2015holder}, who proven that the algorithmic construction of the set of Brascamp--Lieb data whose associated constant is finite is equivalent to the affirmative solution of Hilbert's tenth problem for rational polynomials. It should also be noted that Gowers norms, which have become an object of great interest in additive combinatorics \cite{green2004finite,tao2012higher,bloom2020quantitative}, may be estimated from above via a suitable discrete version of a Brascamp--Lieb inequality. Furthermore, the Brascamp--Lieb inequalities have been found to arise in convex geometry as generalisations of Brunn--Minkowski type inequalities \cite{revbras}, in the study of entropy inequalities for many-body systems of particles \cite{carlen2004sharp}, and have been used as a framework for finding effective solution algorithms for a broad class of optimisation problems arising in computer science \cite{garg2018algorithmic}.

The most immediate question in the theory of linear Brascamp--Lieb inequalities is of course that of finding the necessary and sufficient conditions for BL(\txtbf{L},\txtbf{p}) to be finite. We begin with the observation that, by an elementary scaling argument, the following is a necessary condition for finiteness:
\begin{align}
\sum_{j=1}^mp_jn_j=n.\label{eq:scal}
\end{align}
It was first proven by Barthe, later reproven by Carlen, Lieb, and Loss in \cite{carlen2004sharp}, that this condition together with a spanning condition on the surjections $L_j$ forms a necessary and sufficient condition for finiteness in the rank-one case, i.e. when $n_j=1$ for all $j\in\{1,...,m\}$ \cite{revbras}. The necessary and sufficient conditions for the general case were later established by Bennett, Carbery, Christ, and Tao in two papers via two different methods \cite{bennett2005finite, bcct}, although, in \cite{bcct}, by studying extensively the role of gaussians in the problem, they also obtain additional information about the extremisability of Brascamp--Lieb inequalities.\par
As we shall discuss later on, it can be shown that if a Brascamp--Lieb inequality admits an extremiser, then it must admit a gaussian extremiser, a result that is related to the following theorem due to Lieb.
\begin{thm}[Lieb's Theorem \cite{lieb1990gaussian}]\label{lieb}
Given any Brascamp--Lieb datum $(\txtbf{L},\txtbf{p})$, the set of centred gaussians $\mathcal{G}$,
\begin{align*}
    \mathcal{G}:=\left\lbrace(G_j)_{j=1}^m: G_j(x):=\exp(-\pi\langle A_jx,x\rangle), A_j\in\mathbb{R}^{n_j\times n_j} \textnormal{ is positive definite }\forall j\in\{1,...,m\}\right\rbrace
\end{align*}
exhausts the associated Brascamp--Lieb inequality, that is to say $\sup_{\txtbf{G}\in \mathcal{G}}\BL(\txtbf{L},\txtbf{p};\txtbf{G})=\BL(\txtbf{L},\txtbf{p})$.
\end{thm}
For the remainder of this paper we shall assume that all gaussians are centred. Theorem \ref{lieb} is a very deep and powerful result, since it often allows one to impose without loss of generality that the functions being considered are gaussians, thereby gaining a great deal of structural information. Indeed, it is using this result that necessary and sufficient conditions for both the finiteness and extremisability of Brascamp--Lieb inequalities were proven by Bennett, Carbery, Christ, and Tao in \cite{bcct}. Moreover, in the same paper they establish necessary and sufficient conditions for when such a gaussian extremiser is unique, up to certain translation and scaling invariances we shall soon define. Before we give a statement of their theorem, we shall need to state some preliminary definitions.
\begin{dfn}
Let $(\txtbf{L},\txtbf{p})$ be a Brascamp--Lieb datum. We say that the datum $(\txtbf{L},\txtbf{p})$ is \emph{feasible} if it satisfies the scaling condition (\ref{eq:scal}), and that for all subspaces $V\leq\mathbb{R}^n$,
\begin{align}
\dim(V)\leq\sum_{j=1}^mp_j\dim(L_jV).
\label{eq:feas}
\end{align}
\end{dfn}
We shall use the notation $\mathcal{F}$ to denote the set of $m$-tuples $\txtbf{L}=(L_j)_{j=1}^m$ for which the datum $(\txtbf{L},\txtbf{p})$ is feasible. Note that we have suppressed the implicit dependence of $\mathcal{F}$ on $\txtbf{p}, n, n_1,...,n_m$ due to these objects being essentially fixed, a convention that we shall continue throughout this paper.
\begin{dfn}
 Given $(\txtbf{L},\txtbf{p})$, we say that a proper non-trivial subspace $V\leq\mathbb{R}^n$ is \emph{critical} if it satisfies (\ref{eq:feas}) with equality, and that the datum $(\txtbf{L},\txtbf{p})$ is \textnormal{simple} if it admits no critical subspaces.
\end{dfn}
 Similarly, let $\mathcal{S}$ denote the set of $m$-tuples $\txtbf{L}=(L_j)_{j=1}^m$ such that $(\txtbf{L},\txtbf{p})$ is simple. The significance of critical subspaces is that, if we were to restrict the domains of the surjections $L_j$ to a critical subspace $V$, and their codomains to $L_jV$, then we would obtain a restricted datum that is itself feasible. Moreover, in the case that $(\txtbf{L},\txtbf{p})$ is extremisable, the orthogonal complement of a critical subspace is itself critical \cite{bcct}. As a result, Brascamp--Lieb data can exhibit a certain splitting phenomenon, where they may be decomposed along critical subspaces. In this sense, analogously to the role that simple groups play in group theory, simple Brascamp--Lieb data may be viewed as algebraically fundamental objects in the theory of Brascamp--Lieb inequalities. An in-depth discussion of such structural considerations can be found in \cite{bcct} and \cite{vldm}, for example.
\begin{thm}[Bennett, Carbery, Christ, Tao (2007)\cite{bcct}]\label{bcct}
A Brascamp--Lieb datum $(\txtbf{L},\txtbf{p})$ is finite if and only if it is feasible, and is gaussian-extremisable if it is simple. Moreover, gaussian extremisers for simple data are unique up to translation and scaling invariances.
\end{thm}
To clarify, by `translation and scaling invariance' we are referring to the fact that the Brascamp--Lieb functional $\BL(\txtbf{L},\txtbf{p};\txtbf{f})$ is invariant under the following group actions
$$f_j\mapsto f_j(\cdot - v_j)\quad\text{where}\hspace{3pt} v_j\in\R^{n_j},\qquad f_j\mapsto f_j(\lambda\cdot)\quad\text{where}\hspace{3pt}\lambda>0,$$
and so by saying that there exists a `unique up to invariances' input satisfying a certain property, we mean that there exists a unique orbit under these actions for which that property holds on each input in that orbit. The reader may note that we have omitted the necessary condition for gaussian extremisability from our statement of Theorem \ref{bcct}---roughly speaking, this necessary condition is that there must exist a `direct sum decomposition' of the datum into simple components. It should be made clear that H\"older's inequality, the Loomis--Whitney inequality, and certain cases of Young's convolution inequality are important examples of gaussian-extremisable Brascamp--Lieb inequalities that are associated with non-simple data. \par
Before we can continue with our discussion of the extremisability of linear Brascamp--Lieb inequalities, we shall need to introduce some notation related to gaussians, as it shall often be useful to tailor our notation specifically for them. Let $(\txtbf{L},\txtbf{p})$ be a Brascamp--Lieb datum and let $\txtbf{G}=(G_j)_{j=1}^m$ be an $m$-tuple of gaussians of the form $G_j(x):=\exp(-\pi\langle A_j x,x\rangle)$, where each $A_j\in\R^{n_j\times n_j}$ is a positive-definite matrix. We refer to the $m$-tuple of symmetric positive definite matrices $\txtbf{A}:=(A_j)_{j=1}^m$ as a \textit{gaussian input}, and define the functional $\BLg(\txtbf{L},\txtbf{p};\txtbf{A}):=\BL(\txtbf{L},\txtbf{p};\txtbf{G})$. Of course, since integrals of gaussians may be computed in terms of their underlying matrices, we then have access to the following explicit formula:
\begin{align*}
    \BLg(\txtbf{L},\txtbf{p};\txtbf{A})=\frac{\prod_{j=1}^m\det(A_j)^{p_j/2}}{\det\left(\sum_{j=1}^mp_jL_j^*A_jL_j\right)^{1/2}}
\end{align*}
We will find the regularity of the Brascamp--Lieb constant to be highly relevant to our analysis of the nonlinear case. This subject enjoys its own surprisingly rich theory in the literature; it was Bennett, Bez, Flock, and Lee who first established that the Brascamp--Lieb constant is locally bounded on $\mathcal{F}$ \cite{bennett2018stability}, wherein it may be also be observed that $\mathcal{F}$ is open in $\R^{n_1\times n}\times...\times\R^{n_m\times n}$. This result was later improved to continuity by Bennett, Bez, Cowling, and Flock, who also provide counterexamples showing that the Brascamp--Lieb constant is in general not differentiable.
\begin{thm}[Bennett, Bez, Cowling, Flock (2016) \cite{bennett2017behaviour}]\label{BrasCont}
The mapping $\BL(\cdot,\txtbf{p}):\mathcal{F}\rightarrow\R$ is continuous, but not differentiable.
\end{thm}
Buschenhenke and the above authors later refined this statement in \cite{bennett2020nonlinear}, where they establish that the Brascamp--Lieb constant is locally H\"older continuous on the set of feasible data. In light of the third part of Theorem \ref{bcct}, we know there exists a unique (up to scaling invariance) map $\txtbf{Y}:\mathcal{S}\rightarrow\mathcal{G}$ such that $\BLg(\txtbf{L},\txtbf{p};\txtbf{Y}(\txtbf{L}))=\BL(\txtbf{L},\txtbf{p})$.
\begin{thm}[Valdimarsson (2010) \cite{valdimarsson2011geometric}]\label{vald}
The set $\mathcal{S}$ is open in $\mathcal{F}$, and the map $\txtbf{Y}$ is smooth, whence the Brascamp--Lieb constant is also smooth on $\mathcal{S}$.
\end{thm}
Lieb's theorem tells us that, for any $\delta>0$ and any feasible datum $(\txtbf{L},\txtbf{p})$, there exists a gaussian input $\txtbf{A}$ such that $\BLg(\txtbf{L},\txtbf{p};\txtbf{A})\geq (1-\delta)\BL(\txtbf{L},\txtbf{p})$; we shall refer to such an $\txtbf{A}$ as a \emph{$\delta$-near extremiser} for $(\txtbf{L},\txtbf{p})$. This does not, however, come with any information about the norms and eccentricities of the components of $\txtbf{A}$, nor whether or not this choice of $\txtbf{A}$ may be made differentiably in $\txtbf{L}$ (or even continuously for that matter!). As we shall be dealing with generically non-simple Brascamp--Lieb data, we will be required to prove a certain `effective' version of Lieb's theorem that affords one control on the norms and eccentricities of these $\delta$-near extremisers, playing a similar role to Theorem 1.3 of \cite{bennett2020nonlinear}, but developing on this theorem by additionally asserting that this choice may be made continuously differentiably in $\txtbf{L}\in\mathcal{F}$. This result is formulated as a feasible analogue to Theorem \ref{vald}: for each $\delta>0$, we shall construct a map $\txtbf{Y}_\delta:\mathcal{F}\rightarrow\mathcal{G}$ that sends a given feasible Brascamp--Lieb datum to a $\delta$-near gaussian extremiser for said data, and is such that $\Vert \txtbf{Y}_\delta\Vert_{C^1}$ does not blow up too quickly as $\delta\rightarrow 0$. The construction of this map shall be the content of the forthcoming Theorem \ref{impefflieb}.\par
Our initial exposition of the linear theory now complete, in the next section we turn our attention to the main focus of this paper, this being nonlinear Brascamp--Lieb inequalities.

\subsection{Nonlinear Brascamp--Lieb Inequalities}
Nonlinear Brascamp--Lieb inequalities are a further generalisation of the linear Brascamp--Lieb inequalities, where the linear surjections $L_j$ are replaced with submersions $B_j:M\rightarrow M_j$ between Riemannian manifolds. Given an $m$-tuple of exponents $\txtbf{p}=(p_j)_{j=1}^m$, we shall consider the corresponding inequality:
\begin{align*}
\int_M\prod_{j=1}^m(f_j\circ B_j)^{p_j}\leq C\prod_{j=1}^m\left(\int_{M_j}f_j\right)^{p_j}.
\end{align*}
We shall refer to the pair $(\txtbf{B},\txtbf{p})$ as a \emph{nonlinear Brascamp--Lieb datum} (furthermore we shall refer to such a datum as being $C^k$ for some $k\in\N$ if each of the submersions $B_j$ are $C^k$). Inequalities of this type arise quite naturally in PDE and Fourier restriction contexts, as evidenced in \cite{koch2015convolution,bejenaru2010convolution,bejenaru2011convolution} and \cite{ bennett2010some, bennett2017behaviour, bennett2005non} respectively. Pioneering work was done by Bennett, Carbery, and Wright in \cite{bennett2005non}, wherein they established a nonlinear $C^3$ perturbation of the Loomis--Whitney inequality using the method of refinements due to Christ. This was later improved to $C^{1,\alpha}$-regularity by Bennett and Bez in \cite{bennett2010some} via an induction-on-scales strategy based on work of Bejenaru, Herr, and Tataru \cite{bejenaru2010convolution}, and may itself be viewed as a precursor to the `tight induction-on-scales' methodology developed in \cite{bennett2020nonlinear}. The nonlinear Loomis--Whitney inequality with $C^1$-regularity was later established by Carbery, H\"anninen, and Valdimarsson via multilinear factorisation methods \cite{carbery2022multilinear}. At a similar time to the writing of \cite{bennett2005non}, Bennett, Carbery, and Tao proved a curvilinear multilinear Kakeya inequality (which may be viewed as a more general form of a nonlinear Loomis--Whitney inequality) in \cite{bct}, a celebrated work that has had a substantial influence on harmonic analysis. This result was later generalised to the full Kakeya--Brascamp--Lieb setting by Bennett, Bez, Flock, and Lee \cite{bennett2018stability}, wherein they also prove a Sobolev estimate on nonlinear Brascamp--Lieb forms. Significant progress on the topic of nonlinear Brascamp--Lieb inequalities was made by Bennett, Bez, Buschenhenke, Cowling, and Flock in \cite{bennett2020nonlinear}, where they employ a tight induction-on-scales method that utilises techniques from convex optimisation to prove the following very general local nonlinear Brascamp--Lieb inequality.

\begin{thm}[Local Nonlinear Brascamp--Lieb Inequality (2018)\cite{bennett2020nonlinear}]\label{bbbfc}
Let $\epsilon>0$, and suppose that $(\txtbf{B},\txtbf{p})$ is a $C^2$ nonlinear Brascamp--Lieb datum defined over some neighbourhood $\widetilde{U}$ of a point $x_0\in\R^n$. There exists a neighbourhood $U\subset\widetilde{U}$ of $x_0$ such that the following inequality holds for all non-negative $f_j\in L^1(\R^{n_j})$:
\begin{align}
    \int_{U}\prod_{j=1}^mf_j\circ B_j(x)^{p_j}dx\leq(1+\epsilon)\BL(\txtbf{dB}(x_0),\txtbf{p})\prod_{j=1}^m\left(\int_{\mathbb{R}^{n_j}}f_j\right)^{p_j}.\label{eq:nbras}
\end{align}
\end{thm}
It is natural to ask the question of whether or not there holds a stronger formulation of Theorem \ref{bbbfc} that omits the $(1+\epsilon)$ factor, since we know that certain sharp results hold on the sphere, as established by Carlen, Lieb and Loss in \cite{carlen2004sharp}, later generalised to the setting of compact homogeneous spaces by Bramati in \cite{bramati2019brascamp}, however such a result at the level of generality of \eqref{eq:nbras} is likely beyond the reach of the induction-on-scales methods the authors employ. Before we move onto the topic of global nonlinear Brascamp--Lieb inequalities, we shall mention that there are some other interesting results for compact domains that depart from the usual transversality assumptions of the aforementioned authors, requiring a curvature condition well. This includes $L^p$-improving estimates for multilinear Radon-like transforms, explored by Tao and Wright in the bilinear setting in \cite{tao2003} then generalised by Stovall to the fully multilinear setting in \cite{stovall2011p}. While the topic of Brascamp--Lieb inequalities with curvature conditions is fascinating, we shall not investigate considerations of this type in this paper.\par

While some of the central questions of the local theory of nonlinear Brascamp--Lieb inequalities have been well-addressed, in the global setting many interesting questions remain open, and at this stage almost all known results require rigid structural assumptions. Examples of known global nonlinear Brascamp--Lieb inequalities include one for data that is homogeneous of degree one \cite{homo}, a global weighted nonlinear Loomis--Whitney inequality in $\R^3$ \cite{koch2015convolution}, a weighted nonlinear Brascamp--Lieb inequality for data admitting a certain type of algebraic structure \cite{duncan2021algebraic}, which is based on the Kakeya--Brascamp--Lieb inequalities of Zhang \cite{zhang2017endpoint} and Zorin--Kranich \cite{zorin2019kakeya}, and some results in the context of integration spaces \cite{carl}. This paper represents a small step towards a general theory of global nonlinear Brascamp--Lieb inequalities, and the author hopes that the results herein would serve as useful tools for future study on this topic.

\subsection{Heat-flow Monotonicity}
Establishing that an inequality enjoys some sort of monotonicity property under heat-flow is the basis of many effective proof strategies in a variety of contexts. Schematically, the manner in which such a strategy works is that if one wishes to prove an inequality of the form $A(f)\leq B(f)$ for all $f$ in some class of functions, where $A$ and $B$ are functionals defined on this class, it is enough to prove that there exists a semigroup $S^t$ acting on this class such that $A(f)\leq\underset{t\rightarrow0}{\liminf}A(S^tf)$, $A(S^t f)$ is increasing in $t$, and that $\limsup_{t\rightarrow\infty}A(S^tf)\leq B(f)$. Carlen, Lieb and Loss exploit heat-flow monotonicity to great effect in their proof of the rank-one case of the Brascamp--Lieb inequality \cite{carlen2004sharp}, generalisations of which can be found in, for example, \cite{bcct,bramati2019brascamp}. Heat-flow techniques were also used by Bennett, Carbery and Tao to great effect in their treatment of the multilinear Kakeya and restriction problems \cite{bct}, later generalised by Tao in \cite{tao2020sharp}. Methods that exploit heat-flow monotonicity are often also referred to as `semigroup interpolation' methods (see an article of Ledoux for further reading \cite{ledoux2014heat}), and a systematic study of the generation of monotone quantities for the heat equation can be found in \cite{bennett2015generating}.
An interesting manifestation of heat-flow monotonicity for the Brascamp--Lieb functional arises from the following inequality due to Keith Ball.  
\begin{lem}[Ball's inequality \cite{barthe1998optimal,bennett2010some}]\label{ballsineq}
Let $(\txtbf{L},\txtbf{p})$ be a Brascamp--Lieb datum and let $\txtbf{f}=(f_j)_{j=1}^m, \txtbf{g}=(g_j)_{j=1}^m\in\prod_{j=1}^mL^1(\mathbb{R}^{n_j})$. Given $x\in\R^n$, we define $\txtbf{h}^x:=(f_j(\cdot)g_j(L_j(x)-\cdot))_{j=1}^m$. For all choices of inputs $\txtbf{f}$ and $\txtbf{g}$, the following inequality holds.
\begin{align*}
\BL(\txtbf{L},\txtbf{p};\txtbf{f})\BL(\txtbf{L},\txtbf{p};\txtbf{g})\leq\underset{x\in\mathbb{R}^n}{\sup}\BL(\txtbf{L},\txtbf{p};\txtbf{h}^x)\BL(\txtbf{L},\txtbf{p};\txtbf{f}\ast\txtbf{g})
\end{align*}
\end{lem}
 If we assume that $\BL(\txtbf{L},\txtbf{p})<\infty$ and that $\txtbf{g}$ is an extremising input, i.e. $\BL(\txtbf{L},\txtbf{p};\txtbf{g})=\BL(\txtbf{L},\txtbf{p})$, then this inequality implies the following two statements:
\begin{align}
\BL(\txtbf{L},\txtbf{p};\txtbf{f})&\leq\BL(\txtbf{L},\txtbf{p};\txtbf{f}\ast\txtbf{g}) \label{eq:ball1}\\
\nonumber &\\
\BL(\txtbf{L},\txtbf{p};\txtbf{f})&\leq\underset{x\in\mathbb{R}^n}{\sup}\BL(\txtbf{L},\txtbf{p};\txtbf{h}^x) \label{eq:ball2}
\end{align}
An important consequence is that, if we further suppose that $\txtbf{f}$ is an extremiser, then \eqref{eq:ball1} implies that the set of extremisers is closed under convolution. This, together with the the topological closure of extremisers and the scale-invariance of linear Brascamp--Lieb inequalities, guarantees the existence of a gaussian extremiser given the existence of at least one extremiser, as we may convolve a given extremiser with itself iteratively, then apply the central limit theorem to a rescaled version of the resulting sequence to find that the limiting extremiser must be gaussian \cite{bcct}.\par
Suppose that $\txtbf{g}$ is a gaussian extremiser, and define its associated family of rescalings as $\txtbf{g}_{\tau}:=(\tau^{-n_j/2}g_j(\tau^{-1/2}x))_{j=1}^m$ where $\tau>0$. By the scale-invariance of the Brascamp--Lieb inequality, each $\txtbf{g}_{\tau}$ is also an extremiser, hence if we now substitute $\txtbf{g}_{\tau}$ into (\ref{eq:ball1}), we then see that (\ref{eq:ball1}) states that the Brascamp--Lieb functional is monotone increasing as the inputs flow under the following diffusion equation:
\begin{equation*}
    \partial_tf_j = \nabla\cdot(A_j^{-1}\nabla f_j)
\end{equation*}
where $A_j$ is the positive definite matrix such that $g_j:=\exp(-\pi\langle A_j x,x\rangle)$.
We shall now run the scheme outlined at the beginning of this section to derive the extremisability of the Brascamp--Lieb inequality from (\ref{eq:ball1}), as was carried out in a special case in \cite{bcct}.

\begin{lem}
Let $(\txtbf{L},\txtbf{p})$ be a Brascamp--Lieb datum and assume that \eqref{eq:scal} holds.\\
Let $\txtbf{g}(x):=(g_j(x))_{j=1}^m:=(\exp(-\pi\langle A_j x,x\rangle))_{j=1}^m$ for all $x\in\R^{n_j}$, where $A_j\in\R^{n_j\times n_j}$ is positive definite.\par
If \eqref{eq:ball1} holds for all inputs $\txtbf{f}$, then $\txtbf{g}$ extremises the Brascamp--Lieb inequality associated to $(\txtbf{L},\txtbf{p})$, i.e. $\BL(\txtbf{L},\txtbf{p};\txtbf{g})=\BL(\txtbf{L},\txtbf{p})$.
\end{lem}
\begin{proof}
By homogeneity and scale-invariance of the Brascamp--Lieb functional, we may assume without loss of generality that $\int_{\R^{n_j}}f_j=1$ and $\int_{\R^{n_j}}g_j=1$ for each $j\in\{1,...,m\}$. Given $\tau>0$, we define an anisotropic heat kernel $\txtbf{g}_\tau$ as follows: $$\txtbf{g}_{\tau}(x):=(g_{j,\tau}(x))_{j=1}^m=(\tau^{-n_j/2}\exp(-\pi\tau^{-1}\langle A_jx,x\rangle))_{j=1}^m.$$ Observe that for all $\tau>0$,
\begin{align*}
    \tau^{n_j/2}f_j\ast g_{\tau,j}(L_j(\tau^{1/2} x))&=\int_{\R^{n_j}}f_j(z)\exp(-\pi\tau^{-1}\langle A_j(\tau^{1/2} L_j(x) -z),\tau^{1/2} L_j(x) -z\rangle)dz\\
    &=\int_{\R^{n_j}}f_j(z)\exp(-\pi\vert A_j^{1/2} L_j(x)\vert^2+2\pi\tau^{-1/2}\langle A_jL_j(x),z\rangle -\pi\tau^{-1}\vert z\vert^2)dz\\
    &\underset{\tau\rightarrow\infty}{\longrightarrow}\exp(-\pi\vert A_j^{1/2} L_j(x)\vert^2)\int_{\R^{n_j}}f_j(z)dz=g_j\circ L_j(x).
\end{align*}
Combining this limit with \eqref{eq:ball1} via the dominated convergence theorem then gives us that
\begin{align*}
 \BL(\textbf{L},\textbf{p};\txtbf{f})=\int_{\R^n}\prod_{j=1}^mf_j\circ L_j(x)^{p_j}dx&\leq\int_{\R^n}\prod_{j=1}^m(f_j\ast g_{j,\tau})\circ L_j(x)^{p_j}dx\\
&= \tau^{n/2}\int_{\R^n}\prod_{j=1}^m(f_j\ast g_{j,\tau})\circ L_j(\tau^{1/2} x)^{p_j}dx\\
&=\int_{\R^n}\prod_{j=1}^m \tau^{p_jn_j/2}(f_j\ast g_{j,\tau})\circ L_j(\tau^{1/2} x)^{p_j}dx\\
&\underset{\tau\rightarrow\infty}{\longrightarrow}\int_{\R^n}\prod_{j=1}^m g_j\circ L_j(x)^{p_j}dx\\
&=\BL(\textbf{L},\textbf{p};\textbf{g}).\\
\end{align*}
Taking the supremum over all $\txtbf{f}$ then implies that $$\BL(\txtbf{L},\txtbf{p})=\sup_{\txtbf{f}}\BL(\txtbf{L},\txtbf{p};\txtbf{f})\leq\BL(\txtbf{L},\txtbf{p};\txtbf{g})\leq\BL(\txtbf{L},\txtbf{p}),$$
hence $\textbf{g}$ is an extremiser as claimed.
\end{proof}
Observing this equivalence between heat-flow monotonicity and extremisability, it is then natural to consider whether or not, for some suitable choice of nonlinear Brascamp--Lieb datum, there exists a variable coefficient heat-flow for which the associated nonlinear Brascamp--Lieb functional is monotone, and if so whether or not this would imply that the inequality holds with finite constant. Indeed, this is the approach that was taken in both \cite{bramati2019brascamp} and \cite{carlen2004sharp} to prove nonlinear Brascamp--Lieb inequalities in certain geometrically symmetric settings, so it is then plausible to suppose that a generalisation of such a monotonicity property could hold more broadly. 
The inequalities (\ref{eq:ball1}) and (\ref{eq:ball2}) express an amenability of the linear Brascamp--Lieb functional to two distinct processes, the former being smoothing via heat-flow and the latter being localisation via gaussian extremisers, as we may think of $h_j^x$ as an essentially truncated version of $f_j$, whose essential support is contained within a ball centred at $L_j(x)$. The proof strategy of \cite{bennett2020nonlinear} was to find a nonlinear version of \eqref{eq:ball2} that would serve as a way to bound the left-hand side of \eqref{eq:nbras} above by a supremum of similar integrals over smaller domains, so that if used recursively this would form the engine of an induction-on-scales argument. In this paper we establish a corresponding nonlinear version of \eqref{eq:ball1}, although admittedly we only establish heat-flow near-monotonicity for small times. At its core it is still an induction-on-scales argument, where we tightly bound the possible error between times that are close to one another so that when we string these inequalities together we are left with an error that is well-controlled.\par
This paper was funded by a grant from the EPSRC, and forms part of the author's PhD thesis. She would like to thank her supervisor Jonathan Bennett for his guidance and patience, without which this work could not have been produced, Fred Lin for his many insightful comments, and the anonymous referee for their thorough and detailed feedback as well as various helpful suggestions that have greatly improved the quality of this work.
\section{Main Section}
\subsection{Setup and Notation}\label{setup section}
In this paper, we shall consider fixed Riemannian manifolds $M$, $M_1$,...,$M_m$ (without boundary) of dimensions $n$, $n_1$,...,$n_m$. We shall refer to the exponential map based at a point $x$ on a manifold $N$ by $e_x:T_xN\rightarrow N$. The injectivity radius of a point $x\in N$ is the largest number $\rho_x>0$ such that $e_x$ restricts to a diffeomorphism on the ball of radius $\rho_x$ around $0\in T_x N$. We shall assume that the manifolds we consider have bounded geometry, by which we mean that they have injectivity radii uniformly bounded below, by a number $\rho>0$ which we now fix, and also that both the Riemannian curvature tensor and its first-order covariant derivatives have their norms uniformly bounded from above. For further reading about analysis on manifolds with bounded geometry, see \cite{schick2001manifolds,hebey2000nonlinear}\par
We shall refer to a ball centred at a point $x\in M$ of radius $r>0$ on a manifold $M$ by $U_r(x)$, and refer to a ball centred at a point $v\in T_x M$ of radius $r>0$ by $V_r(v)$ (the tangent space that this ball belongs to should always be clear from context, if it is not stated explicitly). We shall consider submersions $B_j:M\rightarrow M_j$ $(j\in\{1,...,m\})$ that may be viewed as fixed for the entirety, and are assumed to have at least $L^\infty$ bounded derivative maps. Noting this, we shall denote a ball centred at $z\in M_j$ of radius $r\Vert dB_j\Vert_{L^\infty}$ by $U_{r,j}(z)$, and similarly a ball centred at $w\in T_zM_j$ of radius $r\Vert dB_j\Vert_{L^\infty}$ by $V_{r,j}(w)$, simply for the technical reason that then $\bigcap_{j=1}^mdB_j(x)^{-1}(V_{r,j}(0))\subset V_r(0)$, a property that shall prove to be useful later on. We shall also make use of a fixed parameter $\gamma\in(0,1)$ close to $1$. The exact choice of value here is not particularly important, the reader may take $\gamma$ to be $0.9$, say, however we refrain from doing this for the sake of clarity and good book-keeping.\par

We shall always use a single bar to denote a finite dimensional norm, usually a $2$-norm, and double bars to denote an infinite dimensional norm, which we shall always specify with a subscript. In the case where we are taking a norm of a matrix, we shall assume that this is the induced 2-norm unless stated otherwise. Since we may assume that the underlying dimensions, exponents, manifolds, and the nonlinear Brascamp--Lieb datum $(\txtbf{B},\txtbf{p})$ are all fixed throughout, we shall use the relation $A\lesssim B$ to denote that there exists a constant $C>0$ depending only on these objects such that $A\leq C B$, and the relation $A\simeq B$ to denote that $A\lesssim B\lesssim A$. Similarly, if $y$ is some variable, $Q$ is a normed space valued function of $y$, and $f$ is a real valued function of $y$, then we shall use the the notation $Q(y)=\mathcal{O}(f(y))$ to denote that $\Vert Q(y)\Vert\lesssim f(y)$.
\subsection{Statements of Results}\label{results section}

Before we state our nonlinear version of (\ref{eq:ball1}), we must first define our `heat-flow'. The construction thereof is rather involved, however the resulting flow operator $H_{x,\tau,j}$ may nonetheless be written essentially as a convolution with a gaussian kernel $G_{x,\tau,j}:T_{B_j(x)}M_j\rightarrow\R$, the key properties of which we now state as a proposition.
\begin{prop}\label{Gprop}
Suppose that $(\txtbf{B},\txtbf{p})$ is a $C^2$ nonlinear Brascamp--Lieb datum such that \\$\Vert\txtbf{dB}\Vert_{C^1},\Vert\BL(\txtbf{dB},\txtbf{p})\Vert_{L^\infty}\lesssim 1$. Then, there exists an $\epsilon>0$ such that, for $\tau>0$ sufficiently small, there exists a smooth family of $m$-tuples of gaussians $\txtbf{G}_{x,\tau}:=(G_{x,\tau,j})_{j=1}^m$ parametrised by $x\in M$ satisfying the following properties:
\begin{enumerate}
    \item Each gaussian $G_{x,\tau,j}$ is of unit mass and is defined by a corresponding $\tau$-dependent positive definite matrix $A_{\tau,j}(x)$, in the sense that $$G_{x,\tau,j}(z):=\tau^{-n_j}\det(A_{\tau,j}(x))^{1/2}\exp(-\pi\tau^{-2}\langle A_{\tau,j}(x)z,z\rangle).$$
    \item $\txtbf{G}_{x,\tau}$ is a $\tau^\epsilon$-near extremiser for the datum $(\txtbf{dB}(x),\txtbf{p})$.
    \item $\Vert A_{\tau,j}\Vert_{C^1},\Vert \det A_{\tau,j}\Vert_{C^1},\Vert A_{\tau,j}^{-1}\Vert_{L^\infty}\leq\tau^{-\epsilon}$ for all $j\in\{1,...,m\}$.
\end{enumerate}
\end{prop}
Construction of this $G_{x,\tau,j}$ is carried out in Section \ref{impeffliebsection}, from which we observe that property (1) holds immediately; see the end of Section \ref{defG} and the remark after Lemma \ref{abounds} for the proof of properties (2) and (3) respectively. We may now define the corresponding flow operator, wherein we include some truncation to allow us to map locally to the tangent space on which $G_{x,\tau,j}$ is defined.
\begin{align*}
    H_{x,\tau,j}&:L^1(M_j)\rightarrow L^1(U_{\rho-\tau^{\gamma}}(B_j(x)))\\
    H_{x,\tau,j}f_j(z)&:=\int_{U_{\tau^{\gamma},j}(z)}f_j(w)G_{x,\tau,j}(e_{B_j(x)}^{-1}(z)-e_{B_j(x)}^{-1}(w))dw
\end{align*}
We now state our near-monotonicity result, which is the main theorem of this paper.
\begin{thm}[Nonlinear Ball's Inequality]\label{NBall1}
 Suppose that $(\txtbf{B},\txtbf{p})$ is a $C^2$ nonlinear Brascamp--Lieb datum such that $\Vert \txtbf{dB}\Vert_{C^1}, \Vert\BL(\txtbf{dB},\txtbf{p})\Vert_{L^\infty}\lesssim 1$.\par
 Let $U\subset M$ be an open subset of $M$ separated from $\partial M$ (i.e. $\dist(U,\partial M)>0$) such that, for each $j\in\{1,...,m\}$, $B_j(U_j)\subset M_j$ is also separated from $\partial M_j$. Then, there exists a $\beta>0$ such that for $\tau>0$ sufficiently small, for all non-negative $f_j\in L^1(M_j)$,
\begin{align}
    \int_U\prod_{j=1}^mf_j\circ B_j(x)^{p_j}dx\leq (1+\tau^{\beta})\int_{U+U_{\tau^\gamma}(0)}\prod_{j=1}^m H_{x,\tau,j}f_j\circ B_j(x)^{p_j}dx, \label{eq:Nball1}
\end{align}
where $U+U_{\tau}(0)$ denotes the $\tau$-neighbourhood of $U$.
\end{thm}
Of course, in the euclidean case we may identify our domain with each tangent space, and so (\ref{eq:Nball1}) then takes the following more familiar form:
\begin{align*}
    \int_U\prod_{j=1}^mf_j\circ B_j(x)^{p_j}dx\leq (1+\tau^{\beta})\int_{U+U_{\tau^\gamma}(0)}\prod_{j=1}^m f_j\ast (G_{x,\tau,j}\chi_{U_{\tau^{\gamma},j}(0)})\circ B_j(x)^{p_j}dx.
\end{align*}
which of course implies a non-truncated, genuine heat-flow near-monotonicity statement.
\begin{align*}
    \int_U\prod_{j=1}^mf_j\circ B_j(x)^{p_j}dx\leq (1+\tau^{\beta})\int_{U+U_{\tau^\gamma}(0)}\prod_{j=1}^m f_j\ast G_{x,\tau,j}\circ B_j(x)^{p_j}dx.
\end{align*}
It is often the case that nonlinear Brascamp--Lieb inequalities enjoy sufficient diffeomorphism-invariance that to prove them in manifold settings that one may reduce to the euclidean setting through a partition of unity \cite{bennett2020nonlinear} or limiting argument \cite{duncan2021algebraic}. However, in our case, it is necessary that we work explicitly in the abstract manifold setting, as this inequality unfortunately is not sufficiently diffeomorphism-invariant to be reducible to the euclidean case, even in the case where $M\cong\R^n$.\par
The main upshot of Theorem \ref{NBall1} is that one may use the local-constancy of $H_{x,\tau,j}f_j$ to perturb the argument in the right-hand side of (\ref{eq:Nball1}). This may be done at small scales, as in Corollary \ref{cor1}, which is an improvement of Theorem \ref{bbbfc} in the sense that there is now an explicit dependence between the error factor and the size of the neighbourhood.
\begin{cor}\label{cor1}
If $(\txtbf{B},\txtbf{p})$ is a nonlinear Brascamp--Lieb datum satisfying the same conditions as in Theorem \ref{NBall1}, then there exists a $\beta>0$ such that for each $x_0\in M$ and all $\tau>0$ sufficiently small,
\begin{align*}
\int_{U_{\tau}(x_0)}\prod_{j=1}^mf_j\circ B_j(x)^{p_j}dx\leq(1+\tau^{\beta})\BL(\txtbf{dB}(x_0),\txtbf{p})\prod_{j=1}^m\left(\int_{M_j} f_j\right)^{p_j}
\end{align*}
\end{cor}
We should remark at this point that it is likely one may also be able to derive the above corollary from a careful inspection of the proof of Theorem \ref{bbbfc} in \cite{bennett2020nonlinear}, however we nonetheless include it here given that it is a fairly immediate consequence of Theorem \ref{NBall1}. One may also use Theorem \ref{NBall1} to perturb at large scales, from which one may derive that the finiteness of the optimal constant for the inequality associated to a nonlinear Brascamp--Lieb datum is stable under bounded perturbations of this datum, given certain regularity hypotheses.
\begin{cor}\label{cor2}
Consider two nonlinear Brascamp--Lieb data $(\txtbf{B},\txtbf{p})$ and $(\widetilde{\txtbf{B}},\txtbf{p})$ with domain $M$ and codomains $M_1,...,M_m$. Let $d_{M_j}$ denote the natural distance metric on $M_j$. Suppose the following conditions hold:
\begin{itemize}
    \item $(\txtbf{B},\txtbf{p})$ satisfies that hypotheses of Theorem \ref{NBall1},
    \item The inequality associated with $(\widetilde{\txtbf{B}},\txtbf{p})$ holds with finite constant,
    \item $\sup_{x\in M}(d_{M_j}(B_j(x),\widetilde{B}_j(x)))<\rho$ for each $j\in\{1,...,m\}$.
\end{itemize}
Then, the inequality associated with $(\txtbf{B},\txtbf{p})$ holds with finite constant.
\end{cor}
The trade-off in the above corollary is that that when we perturb the nonlinear Brascamp--Lieb datum at large scales we lose all quantitative information about the optimal constant for the perturbed datum, however it is likely that, under small perturbations, one may be able to derive quantitative bounds; we leave the details of this to the interested reader.
\subsection{Reduction of Theorem \ref{NBall1}}\label{schemsection}
Suppressing the dependence on $U$, let $C(s,t)$ denote the best constant $C\in(0,\infty]$ for the following inequality.
\begin{align}
 \int_{U+U_{s^\gamma}(0)}\prod_{j=1}^m H_{x,s,j}f_j\circ B_j(x)^{p_j}dx\leq C(s,t)\int_{U+U_{t^\gamma}(0)}\prod_{j=1}^m H_{x,t,j}f_j\circ B_j(x)^{p_j}dx\label{eq:schem}
\end{align}
It is easy to see that $C(s,t)$ enjoys the submultiplicative property $C(r,t)\leq C(r,s)C(s,t)$. We claim that this observation together with the following proposition, as well as the forthcoming Lemma \ref{pointwconv}, is sufficient to prove Theorem \ref{NBall1}.
\begin{prop}\label{timestep}
    There exist $\beta,\nu>0$ such that, for all $\tau\in(0,\nu)$, $$C(\tau,2^{1/2}\tau)\leq 1+\tau^{\beta}.$$
\end{prop}
\begin{proof}[Proof of Theorem \ref{NBall1} given Proposition \ref{timestep}.]
 Setting $\tau_0=\tau$, define the geometric sequence $\tau_k:=2^{-k/2}\tau_0$ and let $K\in\mathbb{N}$. We can split the constant $C(\tau_K,\tau)$ into pieces that can be dealt with by Proposition \ref{timestep}.
\begin{align*}
	C(\tau_{K},\tau) &\leq C(\tau_{K},\tau_{K-1})C(\tau_{K-1},\tau)\\
   &\leq C(\tau_K,\tau_{K-1})C(\tau_{K-1},\tau_{K-2}) C(\tau_{K-2},\tau)\\
   &\leq...\leq\prod_{k=1}^{K} C(\tau_{k},\tau_{k-1})\leq\prod_{k=1}^{K} (1+\tau_{k}^{\beta})
\end{align*}
Taking logarithms of the above inequality, we obtain that
\begin{align*}
\log(C(\tau_{K},\tau))&\leq\sum_{k=1}^{K}\log(1+\tau_{k}^{\beta})\\
&\leq\sum_{k=1}^{\infty}\tau_{k}^{\beta}=\frac{\tau^{\beta}}{2^{\beta/2}-1}
\end{align*}
It then follows that, making $\tau$ accordingly smaller if necessary, that
$C(\tau_{K},\tau)\leq\exp(\frac{\tau^{\beta}}{2^{\beta/2}-1})\leq(1+\tau^{\beta/2})$. For each $j\in\{1,...,m\}$, let $f_j\in C_0^{\infty}(M_j)$ be a non-negative function. By the forthcoming Lemma \ref{pointwconv}, we know that $H_{x,\tau,j}f_j\circ B_j(x)\rightarrow f_j\circ B_j(x)$ as $\tau\rightarrow 0$ for all $x\in M$, hence we may apply Fatou's lemma, taking the limit as $K\rightarrow\infty$, and the claim then quickly follows from the definition of $C(s,t)$ with $s=\tau_K$ and $t=\tau$, .
\begin{align}
\int_U\prod_{j=1}^mf_j\circ B_j(x)^{p_j}dx&\leq\underset{K\rightarrow\infty}{\liminf}\int_{U+U_{\tau_K^\gamma}(0)}\prod_{j=1}^mH_{x,\tau_K,j}f_j\circ B_j(x)^{p_j}dx\nonumber\\
&\leq\underset{K\rightarrow \infty}{\liminf }C(\tau_K,\tau)\int_{U+U_{\tau^\gamma}(0)}\prod_{j=1}^mH_{x,\tau,j}f_j\circ B_j(x)^{p_j}dx\nonumber\\
&\leq (1+\tau^{\beta/2})\int_{U+U_{\tau^\gamma}(0)}\prod_{j=1}^mH_{x,\tau,j}f_j\circ B_j(x)^{p_j}dx
\end{align}
This implies the theorem since we may extend this inequality by density to general non-negative $f_j\in L^1(M_j)$.
\end{proof}
This initial reduction complete, we now turn our attention to the task of constructing the family of near-extremising gaussians $G_{x,\tau,j}$, but in order to do this we shall first need to establish a slight improvement of the effective version of Lieb's theorem first proven in \cite{bennett2020nonlinear}.

\section{A Regularised Effective Lieb's Theorem}\label{impeffliebsection}

An issue with constructing a suitable heat-flow outside of the case where $(\txtbf{dB}(x),\txtbf{p})$ is simple is that we do not then have a natural choice of gaussian extremiser to use as our heat kernel, in fact, generally speaking $(\txtbf{dB}(x),\txtbf{p})$ may not admit a gaussian extremiser at all. While Lieb's theorem does guarantee the existence of a $\delta$-near gaussian extremiser for any $\delta>0$, i.e. there exists a gaussian input $\txtbf{A}$ such that $\BLg(\txtbf{dB}(x),\txtbf{p};\txtbf{A})\geq(1-\delta)\BL(\txtbf{dB}(x),\txtbf{p})$, it does not offer any quantitative information about this gaussian input. The authors of \cite{bennett2020nonlinear} overcame these problems by establishing an effective version of Lieb's theorem stating that for any given Brascamp--Lieb datum there exists at least one $\delta$-near extremiser that doesn't degenerate too badly as $\delta\rightarrow 0$. We will now give a simplified version of their result.
\begin{thm}[Effective Lieb's theorem\cite{bennett2020nonlinear}]\label{efflieb}
There exists $N\in\N$ depending only on $\txtbf{p},n,n_1,...,n_m$ such that the following holds: For any given $D>0$ there exists $\delta_0>0$ such that for every $\delta \in (0, \delta_0)$ and any feasible datum $(\txtbf{L}, \txtbf{p})$ such that $\BL(\txtbf{L},\txtbf{p}),|\txtbf{L}|\leq D$,
 \begin{align}
     \sup_{\vert A\vert,\vert A^{-1}\vert\leq\delta^{-N}}\BLg(\txtbf{L},\txtbf{p};\txtbf{A})\geq(1-\delta)\BL(\txtbf{L},\txtbf{p}).
 \end{align}
\end{thm}
Recall the definitions of $\mathcal{F}$ and $\mathcal{G}$ from Section \ref{intro}. Theorem \ref{efflieb}, in other words, asserts the existence of a function $\txtbf{Y}_{\delta}^0:\mathcal{F}\rightarrow\mathcal{G}$ such that $\txtbf{Y}_{\delta}^0(\txtbf{L})$ is a $\delta$-near extremiser for $(\txtbf{L},\txtbf{p})$ and both $\Vert\txtbf{Y}_{\delta}^0\Vert_{L^{\infty}}$ and $\Vert(\txtbf{Y}_{\delta}^0)^{-1}\Vert_{L^{\infty}}$ are bounded above by $\delta^{-N}$ (to clarify, $(\txtbf{Y}_{\delta}^0)^{-1}(\txtbf{L})$ refers to the gaussian input whose $j$th entry is the inverse of the $j$th entry of $\txtbf{Y}_{\delta}^0(\txtbf{L})$). It says nothing however about the existence of a smooth, let alone continuous, function with such properties, however our analysis would require $\txtbf{Y}_{\delta}^0$ to be $C^1$ bounded polynomially in $\delta$. The existence of such a regularised version of $\txtbf{Y}_{\delta}^0$ is the content of the following theorem, which is the main result of this section.
\begin{thm}\label{impefflieb}
There exists an $N\in\N$ depending only on $\txtbf{p},n,n_1,...,n_m$ such that the following holds: For all sets $\Omega$ compactly contained in $\mathcal{F}$, there exists a $\nu>0$ such that for all $\delta\in(0,\nu)$, there exists a smooth function $\txtbf{Y}_{\delta}:\Omega\rightarrow\mathcal{G}$ satisfying the following properties:
\begin{enumerate}
    \item For each $\txtbf{L}\in\Omega$, $\txtbf{Y}_{\delta}(\txtbf{L})$ is a $\delta$-near extremiser for $(\txtbf{L},\txtbf{p})$, i.e. \label{impefflieb_p1}$$\BLg(\txtbf{L},\txtbf{p};\txtbf{Y}_{\delta}(\txtbf{L}))\geq(1-\delta)\BL(\txtbf{L},\txtbf{p}),$$
    \item $\Vert\txtbf{Y}_{\delta}\Vert_{C^1(\Omega)},\Vert\txtbf{Y}_{\delta}^{-1}\Vert_{L^{\infty}(\Omega)}\leq \delta^{-N}$. \label{impefflieb_p2}
\end{enumerate}
\end{thm}
In order to ensure property \eqref{impefflieb_p2} in the above theorem holds, we shall need to do some pointwise averaging of $\delta$-near extremisers by using the convenient property that near extremisers are, in a sufficient sense, closed under harmonic addition.
\begin{lem}\label{harmadd}
   Let $(\txtbf{L},\txtbf{p})$ be a feasible Brascamp--Lieb datum and $\delta_1,\delta_2\in(0,1)$. Suppose that $\txtbf{A}_1,\txtbf{A}_2\in\mathcal{G}$ are $\delta_1$-near and $\delta_2$-near extremisers respectively for the datum $(\txtbf{L},\txtbf{p})$, i.e. $\BLg(\txtbf{L},\txtbf{p};\txtbf{A}_i)\geq(1-\delta_i)\BL(\txtbf{L},\txtbf{p})$ for each $i\in\{1,2\}$. Then, their term-wise harmonic addition $(\txtbf{A}_1^{-1}+\txtbf{A}_2^{-1})^{-1}$ is a $(\delta_1+\delta_2)$-near extremiser for $(\txtbf{L},\txtbf{p})$.
\end{lem}
\begin{proof}
Take $\txtbf{A}_1$ and $\txtbf{A}_2$ as above, then by Ball's inequality (Lemma \ref{ballsineq}), we have that
\begin{align}
   && \BLg(\txtbf{L},\txtbf{p};\txtbf{A}_1)\BLg(\txtbf{L},\txtbf{p};\txtbf{A}_2)&\leq\BL(\txtbf{L},\txtbf{p})\BLg(\txtbf{L},\txtbf{p};(\txtbf{A}_1^{-1}+\txtbf{A}_2^{-1})^{-1})\nonumber\\
   &\implies& (1-\delta_1)(1-\delta_2)\BL(\txtbf{L},\txtbf{p})^2&\leq\BL(\txtbf{L},\txtbf{p})\BLg(\txtbf{L},\txtbf{p};(\txtbf{A}_1^{-1}+\txtbf{A}_2^{-1})^{-1})\nonumber\\
  &\implies&  (1-\delta_1)(1-\delta_2)\BL(\txtbf{L},\txtbf{p})&\leq\BLg(\txtbf{L},\txtbf{p};(\txtbf{A}_1^{-1}+\txtbf{A}_2^{-1})^{-1})\nonumber\\
  &\implies&  (1-\delta_1-\delta_2)\BL(\txtbf{L},\txtbf{p})&\leq \BLg(\txtbf{L},\txtbf{p};(\txtbf{A}_1^{-1}+\txtbf{A}_2^{-1})^{-1}) \label{eq:balls}
\end{align}
hence $(\txtbf{A}_1^{-1}+\txtbf{A}_2^{-1})^{-1}$ is a $(\delta_1+\delta_2)$-near extremiser for $(\txtbf{L},\txtbf{p})$ as claimed.
\end{proof}
It is fortunate that the authors of \cite{bennett2020nonlinear} in the same paper establish the H\"older continuity of the Brascamp--Lieb constant as a consequence of their effective Lieb's theorem, since, as it shall turn out, this is an essential ingredient for proving Theorem \ref{impefflieb}.

\begin{prop}[\cite{bennett2020nonlinear}]\label{BLholder}
There exists a number $\theta\in(0,1)$ and a constant $C_0$ depending on the dimensions $(n_j)_{j=1}^m$ and exponents $(p_j)_{j=1}^m$ such that the following holds: Given data $\txtbf{L},\txtbf{L}'$ such that $\vert \txtbf{L}\vert,\vert \txtbf{L}'\vert\leq C_1$ and $\BL(\txtbf{L},\txtbf{p}),\BL(\txtbf{L}',\txtbf{p})\leq C_2$, we then have
\begin{align}
    \vert \BL(\txtbf{L},\txtbf{p})-\BL(\txtbf{L}',\txtbf{p})\vert\leq C_0C_1^{n+\theta(n-1)}C_2^3\vert\txtbf{L}-\txtbf{L}'\vert^{\theta}.\label{eq:BLholder}
\end{align}
\end{prop}

\begin{proof}[Proof of Theorem \ref{impefflieb}.]
The proof strategy is to locally average via harmonic addition the potentially discontinuous function given by Theorem \ref{efflieb} in such a way that we both preserve its good properties and impose on it some additional regularity. We will be averaging values taken on a discrete lattice in $\Omega$, which we shall now define. Let $\theta\in(0,1)$ be an exponent to be determined later, and let $E\subset\Omega$ be the following discrete grid of points:
\begin{align}
   E:=\Omega\cap\left(\left(\frac{\delta}{100}\right)^{\frac{1}{\theta}}\prod_{j=1}^m\Z^{n_j\times n}\right). 
\end{align}
Now, let $\mathcal{I}$ be an indexing set for $E$ so that we may write $E=\{\txtbf{L}_i\}_{i\in\mathcal{I}}$, then let $\mathcal{Q}:=\{Q_i\}_{i\in\mathcal{I}}$ be a cover of $\Omega$ via axis-parallel cubes of width equal to $(\delta/10)^{\frac{1}{\theta}}$, with each $Q_i$ centred at $\txtbf{L}_i$. One should note as a matter of technicality that we may need to take $\delta$ to be very small for $\mathcal{Q}$ to genuinely be a cover of $\Omega$.\par 
By Theorem \ref{efflieb}, there exists an $N\in\N$ such that for sufficiently small $\delta>0$ there exists a function $\txtbf{Y}^0_{\delta}:\Omega\rightarrow\mathcal{G}$ such that $\Vert\txtbf{Y}^0_{\delta}\Vert_{L^{\infty}(\Omega)},\Vert(\txtbf{Y}^0_{\delta})^{-1}\Vert_{L^{\infty}(\Omega)}\leq \delta^{-N}$ and $\txtbf{Y}^0_{\delta}(\txtbf{L})$ is a $\delta/2$-near extremiser for $(\txtbf{L},\txtbf{p})\in\Omega$. We begin by showing that, for a suitable choice of $\theta$ and provided that $\delta$ is chosen to be sufficiently small, for all $i\in\mathcal{I}$, $\textbf{Y}_\delta^0(\txtbf{L}_i)$ is also a $\delta$-near extremiser for any $(\txtbf{L},\txtbf{p})$ such that $\txtbf{L}\in Q_i\cap\Omega$. By the relative compactness of $\Omega$ and smoothness of the Brascamp--Lieb functional in $\txtbf{L}\in\mathcal{F}$ for a given fixed input, there exists a $\nu_1\in (0,1)$ such that for $\eta\in(0,\nu_1)$ and all $\txtbf{L},\txtbf{L}'\in\Omega$ satisfying $\vert\txtbf{L}-\txtbf{L}'\vert\leq\eta^2$, we have
\begin{align}
    \BLg(\txtbf{L}',\txtbf{p};\txtbf{Y}^0_{\delta}(\txtbf{L}))\geq(1-\eta)\BLg(\txtbf{L},\txtbf{p};\txtbf{Y}^0_{\delta}(\txtbf{L})).\label{eq:impefflieb1}
\end{align}
The presence of the exponent $2$ in the upper bound $\eta^2$ is merely for the purposes of absorbing constants. By Proposition \ref{BLholder}, we may choose $\theta\in(0,1/2)$ such that the following holds: There exists $\nu_2\in(0,1)$ such that for $\eta\in(0,\nu_2)$ and $\vert\txtbf{L}-\txtbf{L}'\vert\leq\eta^{\frac{1}{\theta}}$, we have that 
\begin{align}
    \BL(\txtbf{L},\txtbf{p})\geq(1-\eta)\BL(\txtbf{L}',\txtbf{p}).\label{eq:impefflieb2}
\end{align}
Again we have used some freedom in our choice in $\theta$ to absorb the constants that arise in \eqref{eq:BLholder}. Choose $\delta$ such that $0<\delta\leq\min\{\nu_1,\nu_2,1\}$. For all $i\in\mathcal{I}$ and all $\txtbf{L}\in Q_i\cap\Omega$, since $\vert\txtbf{L}-\txtbf{L}_i\vert<\delta^{\frac{1}{\theta}}/10^{\frac{1}{\theta}}\leq \delta^2/100$, we may apply \eqref{eq:impefflieb1} and \eqref{eq:impefflieb2} together with the fact that  $\txtbf{Y}^0_{\delta}(\txtbf{L}_i)$ is a $\delta/2$-near extremiser for $(\txtbf{L}_i,\txtbf{p})$ to prove the claim.
\begin{align}
     \BLg(\txtbf{L},\txtbf{p};\txtbf{Y}^0_{\delta}(\txtbf{L}_i))&\geq(1-\delta/10)\BLg(\txtbf{L}_i,\txtbf{p};\txtbf{Y}^0_{\delta}(\txtbf{L}_i))\nonumber\\
     &\geq(1-\delta/2)(1-\delta/10)\BLg(\txtbf{L}_i,\txtbf{p})\nonumber\\
     &\geq(1-\delta/2)(1-\delta/10)^2\BLg(\txtbf{L},\txtbf{p})\nonumber\\
     &\geq(1-\delta)\BL(\txtbf{L},\txtbf{p})\label{stabofext}
\end{align}
Now, let $\{\rho_i\}_{i\in\mathcal{I}}$ be a smooth partition of unity subordinate to $\mathcal{Q}$ such that $\Vert \nabla\rho_i\Vert_{L^{\infty}}\lesssim\delta^{-\frac{1}{\theta}}$ (this can easily be constructed by translation and rescaling), and define the function $\txtbf{Y}_{\delta}:\Omega\rightarrow\mathcal{G}$.
\begin{align}
    \txtbf{Y}_{\delta}(\txtbf{L}):=\left(\sum_{i\in\mathcal{I}}\rho_i(\txtbf{L})\txtbf{Y}_{\delta}^0(\txtbf{L}_i)^{-1}\right)^{-1}\label{eq:defofa}
\end{align}
Again, we clarify that inversions are defined component-wise. We claim that, for any $\txtbf{L}\in\Omega$, $\txtbf{Y}_{\delta}(\txtbf{L})$ is an $\mathcal{O}(\delta)$-near extremiser for $(\txtbf{L},\txtbf{p})$. Firstly, by the scale-invariance of the Brascamp--Lieb functional and \eqref{stabofext}, each $\rho_i(\txtbf{L})^{-1}\txtbf{Y}_{\delta}^0(\txtbf{L}_i)$ is a $\delta$-near extremiser for all $(\txtbf{L},\txtbf{p})$ such that $\txtbf{L}\in Q_i\cap\Omega$.

Since we are pointwise only ever summing boundedly many contributions in \eqref{eq:defofa}, by iterating \eqref{eq:balls}, we find that $\txtbf{Y}_{\delta}(\txtbf{L})$ is an $\mathcal{O}(\delta)$-near extremiser for $(\txtbf{L},\txtbf{p})$ (similar observations about the closure of extremisers under harmonic addition were made in \cite{bcct}). We may of course remove the implicit constant here by a simple substitution, so we shall proceed assuming that $\txtbf{Y}_{\delta}(\txtbf{L})$ is a $\delta$-near extremiser for $(\txtbf{L},\txtbf{p})$, for all $\txtbf{L}\in\Omega$.\par
It remains to prove that $\txtbf{Y}_{\delta}$ satisfies the necessary $L^\infty$ and $C^1$ bounds. We shall start with the $L^{\infty}$ bounds. One bound is trivial, namely that
\begin{align*}
   |\txtbf{Y}_{\delta}(\txtbf{L})^{-1}|\leq\max_{i: \txtbf{L}\in Q_i}|\txtbf{Y}_{\delta}^0(\txtbf{L}_i)^{-1}|\leq\delta^{-N}.
\end{align*}
    The other requires the elementary fact that, for all positive-definite matrices $A,B\in\R^{n\times n}$, $$|(A^{-1}+B^{-1})^{-1}|\leq|A|.$$
    This holds since if $B^{-1}\succ 0$, we then have $0\prec A^{-1}\prec A^{-1}+B^{-1}$, so
    $$(A^{-1}+B^{-1})^{-1}\prec A$$
    and the claim then follows. Applying this to $\txtbf{Y}_{\delta}(\txtbf{L})$ then yields the desired bound.
    \begin{align*}
        |\txtbf{Y}_{\delta}(\txtbf{L})|\lesssim\max_{i: \txtbf{L}\in Q_i}|\txtbf{Y}_{\delta}^0(\txtbf{L}_i)|\leq\delta^{-N}.
    \end{align*}
Finally, it remains to prove the $L^{\infty}$ bound on the derivative $\txtbf{dY}_{\delta}$. We use the chain rule to deal with the matrix inversions, apply the above established bounds on $|\txtbf{Y}_{\delta}(\txtbf{L})|$, then apply the triangle inequality to show that the derivative is at most polynomially bounded. Taking some unit vector $W\in\prod_{j=1}^m\R^{n_j\times n}$,
\begin{align*}
    |\txtbf{dY}_{\delta}[\txtbf{L}](W)|&=|\txtbf{Y}_{\delta}(\txtbf{L})\txtbf{d}((\txtbf{Y}_{\delta})^{-1})[\txtbf{L}](W)\txtbf{Y}_{\delta}(\txtbf{L})|\\
    &\leq\delta^{-2N}|\txtbf{d}((\txtbf{Y}_{\delta})^{-1})[\txtbf{L}](W)|\\
    &\leq\delta^{-2N}\sum_{i\in\mathcal{I}}|\nabla\rho_i(\txtbf{L})||\txtbf{Y}_{\delta}^0(\txtbf{L}_i)|\\
    &\lesssim\delta^{-3N-\frac{1}{\theta}}.
\end{align*}
Hence $\Vert \txtbf{dY}_{\delta}\Vert_{L^\infty}\leq\delta^{-3N-\frac{1}{\theta}}$, completing the proof.
\end{proof}
\subsection{Definition of \texorpdfstring{$G_{x,\tau,j}$}{Gxtj}}\label{defG}

We shall now define the gaussian arising in the statement of Theorem \ref{NBall1} using Theorem \ref{impefflieb}. In order to do this, we need to find a way of globally applying Theorem \ref{impefflieb} to our manifold context, and to this end, we define $\mathcal{F}_x$ to be the set of $m$-tuples $\txtbf{L}:=(L_j)_{j=1}^m$ consisting of surjections $L_j:T_xM\rightarrow T_{B_j(x)}M_j$ such that $(\txtbf{L},\txtbf{p})$ is a feasible Brascamp--Lieb datum, and we consider the following set
$$\Omega_x:=\{\txtbf{L}\in\mathcal{F}_x:|\txtbf{L}|,\BL(\txtbf{L},\txtbf{p})< C\},$$
where $C\sim 1$ is some large constant. The manifold $\mathcal{F}_M:=\bigsqcup_{x\in M}\mathcal{F}_x$ then defines a fibre bundle over $M$, with natural projection map $\pi_{\mathcal{F}}:\mathcal{F}_M\rightarrow M$. Furthermore, $\Omega_M:=\bigsqcup_{x\in M}\Omega_x$ defines a fibre sub-bundle of $\mathcal{F}_M$, and, provided we choose $C$ to be sufficiently large, it contains the image of the section $\txtbf{dB}:M\rightarrow\mathcal{F}_M$. Let $\mathcal{U}:=\{U_a\}_{a\in\mathcal{A}}$ be a boundedly overlapping cover of $M$ via small balls of the same radius, let $\{\phi_a:U_a\rightarrow\R^n\}_{a\in\mathcal{A}}$ be a normal atlas and $\{\phi_{j,a}:B_j(U_a)\rightarrow\R^{n_j}\}_{a\in\mathcal{A}}$ be an atlas for $B_j(M)$ consisting of restrictions of normal charts. We may use them to define a system of local trivialisations for $\mathcal{F}_M$.
\begin{align*}
    \psi_a&:\pi^{-1}_{\mathcal{F}_M}(U_a)\rightarrow U\times\mathcal{F}\\
    \psi_a(x,\txtbf{L})&:=(x,(d\phi_{j,a}[B_j(x)]\circ L_j\circ d\phi_a[x]^{-1})_{j=1}^m)
\end{align*}
By our bounded geometry assumptions, the forthcoming Lemma \ref{geolem} implies that the exponential map has bounded first and second derivatives, hence our normal atlases may be chosen such that  for all $a\in\mathcal{A}$, $\psi_a(\pi^{-1}_{\mathcal{F}}(U_a)\cap\Omega_M)\subset U_a\times\Omega$, where $\Omega:=\{\txtbf{L}\in\mathcal{F}:|\txtbf{L}|,\BL(\txtbf{L},\txtbf{p})< 2C\}$. The set $\Omega$ is open and relatively compactly contained in $\mathcal{F}$, therefore there exists a $\txtbf{Y}_\delta:\Omega\rightarrow\mathcal{G}$ as in Theorem \ref{impefflieb} for this choice of $\Omega$. Let $\{\rho_a\}_{a\in\mathcal{A}}$ be a partition of unity subordinate to $M$ with uniformly bounded derivatives, and define the following gaussian input-valued function:  
\begin{align}
    \txtbf{a}_\tau(x)&:=\left(\sum_{a\in\mathcal{A}}\rho_a(x)\left(\txtbf{C}_a(x)^*\txtbf{Y}_{\tau^{\alpha}}\circ\pi_2\circ\psi_a(x, \txtbf{dB}(x))\txtbf{C}_a(x)\right)^{-1}\right)^{-1},\label{eq:ataudef}
\end{align}
where $\txtbf{C}_a(x):=(d\phi_{j,a}(B_j(x)))_{j=1}^m$, $\pi_2$ denotes projection onto the second component, and $\alpha\in (0,1)$ is a small exponent to be later determined, which we shall use to control the blow-up of $\txtbf{a}_\tau$ under various norms. By scale-invariance of the Brascamp--Lieb inequality and applying a change of co-ordinates, we can see that each term of the form $\txtbf{C}_a(x)^*\txtbf{Y}_{\tau^{\alpha}}\circ\pi_2\circ\psi_a(x, \txtbf{dB}(x))\txtbf{C}_a(x)$ in \eqref{eq:ataudef} is a $\tau^\alpha$-near extremiser for $(\txtbf{dB}(x),\txtbf{p})$, therefore applying Lemma \ref{harmadd} yields that $\txtbf{a}_{\tau}(x)$ is a $\mathcal{O}(\tau^\alpha)$-near extremiser for $(\txtbf{dB}(x),\txtbf{p})$. Moreover, following the same reasoning as in the proof of Theorem \ref{impefflieb}, we may derive that $\Vert \txtbf{a}_\tau\Vert_{L^\infty(M)},\Vert \txtbf{a}_\tau^{-1}\Vert_{L^\infty(M)}\lesssim\tau^{-\alpha N}$. In both instances, we may ignore the implicit constants that arise by simply raising the exponent $\alpha$ by some very small amount. We may then define a gaussian $g_{x,\tau,j}:T_{B_j(x)}M_j\rightarrow\R$ as
\begin{align*}
    g_{x,\tau,j}(z)=\tau^{-n_j}\exp\left(-\frac{\pi}{\tau^2}\langle a_{\tau,j}(x)z,z\rangle\right).
\end{align*}
Implicitly, we may view this gaussian as the fundamental solution of the following anisotropic heat equation at time $t=\tau^2$.
\begin{align*}
    \partial_t u(z,t)=\nabla_z\cdot(a_{\tau,j}(x)^{-1}\nabla_z u(z,t))
\end{align*}
At last, we define our gaussian kernel $G_{x,\tau,j}$ as the following infinite convolution.
\begin{align*}
    G_{x,\tau,j}:=\mathlarger{\mathlarger{\Asterisk}}_{k=1}^{\infty}g_{x,2^{-k/2}\tau,j}
\end{align*}
We shall now show that $G_{x,\tau,j}$ is well-defined if $\alpha<2N^{-1}$, where this $N\in\N$ is the one that arises in Theorem \ref{impefflieb}. To see this, we consider the partial convolution 
\begin{align*}
   G^{(K)}_{x,\tau,j}(v):=g_{x,2^{-1/2}\tau,j}\ast...\ast g_{x,2^{-K/2}\tau,j}(v)=\tau^{-n_j}\det(C_K)^{1/2}\exp(-\pi\tau^2\langle C_Kv,v\rangle) ,
\end{align*}
where $C_K:=(\sum_{k=1}^{k=K}2^{-k}a_{2^{-k/2}\tau,j}(x)^{-1})^{-1}$ (this formula may be checked by an application of the Fourier transform). We now just need to show that $C_K$ converges as $K\rightarrow\infty$, since then $G^{(K)}_{x,\tau,j}$ converges pointwise. Let $l\in\N$, then by the fact that $\Vert a_{2^{-k/2}\tau,j}^{-1}\Vert_{L^\infty(M)}\leq 2^{k\alpha N/2}\tau^{-\alpha N}$ for all $k>0$,
\begin{align*}
    |C_{K+l}^{-1}-C_K^{-1}|&\leq\sum_{k=K+1}^{K+l}2^{-k}|a_{2^{-k/2}\tau,j}(x)^{-1}|\\
    &\leq\sum_{k=K+1}^{K+l}2^{-k}2^{k\alpha N/2}\tau^{-\alpha N}\\
    &\leq2^{(\alpha N/2-1)K}\tau^{-\alpha N}\sum_{k=1}^{l}2^{(\alpha N/2-1)k}
\end{align*}
By our choice of $\alpha$, $|C_{K+l}^{-1}-C_K^{-1}|\rightarrow 0$ as $K\rightarrow\infty$ uniformly in $l$, so $C_K^{-1}$ is cauchy, and therefore converges. By continuity of matrix inversion, the limit of $C_K$ then exists provided that $\lim_{K\rightarrow\infty}(C_K^{-1})\in GL_{n_j}(\R)$, and otherwise $C_K$ must be unbounded, since if it were bounded it would admit a convergent subsequence, which would have to converge to the inverse of the limit of $C_K^{-1}$. We therefore only need to check that $C_K$ is bounded, whence $G^{(K)}_{x,\tau,j}\rightarrow G_{x,\tau,j}$ pointwise, which follows from applying \eqref{harmadd} and the $L^\infty$ bound on $\txtbf{a}_\tau$.
\begin{align*}
    |C_K|\lesssim\sum_{k=1}^K2^{-k}|a_{2^{-k/2}\tau,j}(x)|&\leq\sum_{k=1}^\infty2^{(\alpha N/2-1)k}\tau^{-\alpha N}\lesssim_{\alpha,N}\tau^{-\alpha N}<\infty,
\end{align*}
If we denote the limit of $C_K$ by $A_{\tau,j}(x)$, then we may write $G_{x,\tau,j}(z)$ explicitly as $$G_{x,\tau,j}(z)=\tau^{-n_j}\det(A_{\tau,j}(x))^{1/2}\exp(-\pi\tau^{-2}\langle A_{\tau,j}(x)z,z\rangle).$$
Applying Lemma \ref{harmadd} infinitely many times to $\txtbf{A}_{\tau}(x):=(A_{\tau,j}(x))_{j=1}^m$ we see that this gaussian input is an $\mathcal{O}(\tau^{\alpha})$-near extremiser for $(\txtbf{dB}(x),\txtbf{p})$, hence satisfying property (2) of Proposition \ref{Gprop}.
\begin{align*}
    \frac{\BLg(\txtbf{dB}(x),\txtbf{p};\txtbf{A}_{\tau}(x))}{\BL(\txtbf{L},\txtbf{p})}\geq1-\tau^{\alpha}\sum_{k=1}^\infty2^{-k\alpha/2}=1-\tau^{\alpha}(2^{\alpha/2}-1)^{-1}
\end{align*}
Observe that in the case where $(\txtbf{dB}(x),\txtbf{p})$ is simple, we may forego Theorem \ref{impefflieb} and use an exact extremiser for our definition of $g_{x,\tau,j}$, in which case $\txtbf{a}_{\tau}$ is constant in $\tau>0$, and we would then have the identifications $\txtbf{A}_{\tau}=\txtbf{a}_\tau$ and $G_{x,\tau,j}=g_{x,\tau,j}$. Of course, if the reader were to run our argument in the simple case with exact extremisers, then they would need to take care to ensure that these exact extremisers satisfy appropriate $C^1$ boundedness of the type we shall prove for our near-extremisers in the next section.
\section{Gaussian Lemmas}
This section is, for the most part, dedicated to establishing the properties we require of our gaussians $G_{x,\tau,j}$ and $g_{x,\tau,j}$ in order to prove Proposition \ref{timestep}, which, as we have shown in Section \ref{schemsection}, implies Theorem \ref{NBall1}. We need to quantify how these gaussians behave under small perturbations in a number of variables, and for this purpose we shall first need to prove various bounds on the norms of the underlying gaussian input-valued function $\txtbf{A}_\tau$.

\begin{lem}\label{abounds}
For any $\epsilon>0$, provided $\alpha$ is chosen such that $\alpha<\min\{\frac{2}{3N},\frac{\epsilon}{N}\}$, there exists a $\nu>0$ such that for every $\tau\in(0,\nu)$, $\Vert \txtbf{A}_\tau\Vert_{C^1},\Vert\det\txtbf{A}_\tau\Vert_{C^1}, \Vert \txtbf{A}_\tau^{-1}\Vert_{L^{\infty}}\leq\tau^{-\epsilon}$.
\end{lem}
\begin{proof}
The proof follows quickly from Theorem \ref{impefflieb}, as it amounts to a straightforward application of the triangle inequality and an application of the bounds on $a_{2^{-k/2}\tau,j}(x)$ that immediately follow from Theorem \ref{impefflieb}, taking $\nu>0$ small enough so that we may bound any constants that arise from above by $\tau^{\alpha N-\epsilon}$, for all $\tau\in(0,\nu)$.
\begin{align*}
    |A_{\tau,j}(x)^{-1}|\leq\sum_{k=1}^\infty2^{-k}|a_{2^{-k/2}\tau,j}(x)^{-1}|&\leq\sum_{k=1}^\infty2^{(\alpha N/2-1)k}\tau^{-\alpha N}\leq\tau^{-\epsilon}\\
    |A_{\tau,j}(x)|\lesssim\sum_{k=1}^\infty2^{-k}|a_{2^{-k/2}\tau,j}(x)|&\leq\sum_{k=1}^\infty2^{(\alpha N/2-1)k}\tau^{-\alpha N}\leq\tau^{-\epsilon}\\
\end{align*}
Now, take $W\in T_{B_j(x)}M_j$ such that $|W|=1$, then the bound on $dA_{\tau,j}(x)$ follows from the $L^\infty$ boundedness of $\txtbf{Y}_\delta^{-1}$ and the $C^1$ boundedness of $\txtbf{Y}_\delta$
\begin{align*}
    |dA_{\tau,j}(x)(W)|&=|A_{\tau,j}(x)d(A_{\tau,j}^{-1})(x)(W)A_{\tau,j}(x)|\\
    &\leq\tau^{-2\epsilon}\left|\sum_{k=1}^\infty 2^{-k}d(a_{2^{-k/2}\tau,j})^{-1}(x)(W)\right|\\
    &\lesssim\tau^{-2\epsilon}\sum_{k=1}^\infty 2^{-k}\Vert\txtbf{d}(\txtbf{Y}_{2^{-\alpha k/2}\tau^\alpha}^{-1})\Vert_{L^\infty}\Vert\txtbf{d}^2\txtbf{B}\Vert_{L^\infty}\\
    &\lesssim\tau^{-2\epsilon}\sum_{k=1}^\infty 2^{-k}\Vert\txtbf{Y}^{-1}_{2^{-\alpha k/2}\tau^\alpha}\Vert_{L^\infty}^2\Vert\txtbf{d}\txtbf{Y}_{2^{-\alpha k/2}\tau^\alpha}\Vert_{L^\infty}\\
    &\leq\tau^{-2\epsilon}\sum_{k=1}^\infty 2^{k(3\alpha N/2-1)}\tau^{-2\alpha N}\lesssim\tau^{-4\epsilon}
\end{align*}
We now turn our attention to the $C^1$ bound for $\det(A_{\tau,j}(x))$. First of all, $|\det A_{\tau,j}(x)|\leq|A_{\tau,j}(x)|^{n_j}\leq\tau^{-\epsilon n_j}$ for all $\tau\in(0,\nu)$, so we have the bound $|A_{\tau,j}(x)|\leq\tau^{-\epsilon}$, and similarly $|A_{\tau,j}(x)^{-1}|\leq\tau^{-2\epsilon}$ for all such $\tau$. all that remains is to establish the $L^\infty$ bound on $d(\det A_{\tau,j})$. The case when $n_j=1$ has already been established, since then $\det A_{\tau,j}=A_{\tau,j}$, so suppose then that $n_j>1$. Taking any $x\in M$ and $w\in T_x M$ such that $|w|=1$, by Jacobi's formula, the chain rule, the Cauchy-Schwarz inequality, and the equivalence of finite dimensional norms,
\begin{align*}
   |d(\det A_{\tau,j})[x](W)|&=|\adj(A_{\tau,j}(x))^\ast:dA_{\tau,j}(x)(w)|\\
   &\lesssim|A_{\tau,j}(x)|^{n_j-1}|dA_{\tau,j}(x)|\lesssim\tau^{-2(n_j-1)\epsilon}
\end{align*}
This proves the claim, since we may adjust $\epsilon$ accordingly.
\end{proof}
We shall henceforth consider $\epsilon\in(0,1-\gamma)$ and $\alpha\in(0,2\epsilon /3N)$ as fixed parameters, and we also note at this point that we have now proven property (3) of Proposition \ref{Gprop}, thus completing its proof.
\begin{lem}\label{switchH}
For all $\eta\in(0,\min\{\gamma-2\epsilon,0.9\gamma-\epsilon,3\gamma-2-\epsilon\})$, there exists a $\nu>0$ such that the following holds: for all $\tau\in(0,\nu)$ and $x,y\in M$ such that $d(x,y)\leq\tau^\gamma$, and $z\in M_j$ such that $d(z,B_j(x))\leq\tau^\gamma$, for all $f_j\in L^1(M_j)$,
\begin{align}
H_{y,\tau,j}f_j(z)\leq (1+\tau^{\eta})H_{x,\tau,j}f_j(z)\label{eq:lem2}
\end{align}
\end{lem}
In order to prove this statement, we shall need a fairly routine geometric lemma, the proof of which may be found in the appendix.
\begin{lem}\label{geolem}
Let $M$ be a Riemannian manifold with bounded geometry of injectivity radius $\rho$. Then norms of the covariant derivatives (up to second order) of the exponential map based at $p\in M$ are bounded above on the ball of radius $\rho$ centred at $0$ in $T_pM$, uniformly in $p$.
\end{lem}
\begin{proof}[Proof of Lemma \ref{switchH}.]
Let $\tau>0$ be small, let $x,y\in M$ satisfy $d(x,y)\leq\tau^{\gamma}$, and take some $z\in M_j$ such that $d(z,B_j(x))\leq\tau^\gamma$. First of all, by the chain rule, for any $v\in T_x M$, $d(e_{B_j(y)}^{-1}\circ e_{B_j(x)})[v]=d(e_{B_j(y)}^{-1})[e_{B_j(x)}(v)]de_{B_j(x)}[v]$.
Given $w\in U_{\tau,j}(z)$, by Taylor's theorem, we may approximate $v_y:=e_{B_j(y)}^{-1}(z)-e_{B_j(y)}^{-1}(w)$ in terms of $v_x:=e_{B_j(x)}^{-1}(z)-e_{B_j(x)}^{-1}(w)$ in the following manner:
\begin{align}
    v_y&=e_{B_j(y)}^{-1}\circ e_{B_j(x)}\circ e_{B_j(x)}^{-1}(z)-e_{B_j(y)}^{-1}\circ e_{B_j(x)}\circ e_{B_j(x)}^{-1}(w)\nonumber\\
    &=d(e_{B_j(y)}^{-1}\circ e_{B_j(x)})[e_{B_j(x)}^{-1}(z)](v_x)+\mathcal{O}(|v_x|^2)\nonumber\\
    &=d(e_{B_j(y)}^{-1})[z]de_{B_j(x)}[e_{B_j(x)}^{-1}(z)](v_x)+\mathcal{O}(|v_x|^2).\label{eq:higherorder}
\end{align}
Above, we use Lemma \ref{geolem} to uniformly bound the higher derivatives. Define the linear map $T_{x,y}:=d(e_{B_j(y)}^{-1})[z]de_{B_j(x)}[e_{B_j(x)}^{-1}(z)]$, then it follows that
\begin{align}
    |A_{\tau,j}(y)^{1/2}v_y|^2&=|A_{\tau,j}(y)^{1/2}(T_{x,y}v_x+\mathcal{O}(|v_x|^2))|^2\nonumber\\
    &\leq|A_{\tau,j}(y)^{1/2}(T_{x,y}v_x)|^2+\tau^{2.9\gamma-\epsilon},\label{eq:switchH1}
\end{align}
for sufficiently small $\tau>0$. Now, by the uniform bounds on $\det A_{\tau,j}$ established in Lemma \ref{abounds}, we have that
\begin{align*}
    |\log(\det A_{\tau,j}(x))-\log(\det A_{\tau,j}(y))|&\leq\frac{|\det A_{\tau,j}(x) - \det A_{\tau,j}(y)|}{\min\{|\det A_{\tau,j}(x)|,|\det A_{\tau,j}(y)|\}}\\
    &\leq\tau^{-\epsilon}\Vert d(\det A_{\tau,j})\Vert_{L^\infty}d(x,y)\\
    &\leq\tau^{\gamma-2\epsilon}.
\end{align*}
Together with \eqref{eq:switchH1}, this implies the bound $G_{y,\tau,j}(v_y)\leq(1+\tau^{\eta})G_{x,\tau,j}(v_x)$ for sufficiently small $\tau>0$.
\begin{align*}
   \frac{G_{y,\tau,j}(v_y)}{G_{x,\tau,j}(v_x)}&=\frac{\det(A_{\tau,j}(y))}{\det(A_{\tau,j}(x))}\exp(\pi\tau^{-2}(\vert A_{\tau,j}(x)^{1/2}v_x\vert^2-\vert A_{\tau,j}(y)^{1/2}v_y\vert))\\
   &\leq\exp(\tau^{\gamma-2\epsilon}+\pi\tau^{0.9\gamma-\epsilon}+\pi\tau^{-2}(\vert A_{\tau,j}(x)^{1/2}v_x\vert^2-\vert A_{\tau,j}(y)^{1/2}T_{x,y}v_x\vert^2))\\
   &\leq\exp(\tau^{\gamma-2\epsilon}+\pi\tau^{0.9\gamma-\epsilon}+\pi\tau^{-2}\langle (A_{\tau,j}(x)-T_{x,y}^*A_{\tau,j}(y)T_{x,y})v_x,v_x\rangle)\\
   &\leq\exp(\tau^{\gamma-2\epsilon}+\pi\tau^{0.9\gamma-\epsilon}+\pi\tau^{-2} |A_{\tau,j}(x)-T_{x,y}^*A_{\tau,j}(y)T_{x,y}||v_x|^2)\\
   &\leq\exp(\tau^{\gamma-2\epsilon}+\pi\tau^{0.9\gamma-\epsilon}+2\pi\tau^{-2}\Vert dA_{\tau,j}\Vert_{L^\infty}\tau^{3\gamma})\\
&\leq\exp(\tau^{\gamma-2\epsilon}+\pi\tau^{0.9\gamma-\epsilon}+2\pi\tau^{3\gamma-2-\epsilon})\leq1+\tau^\eta
\end{align*}
In the penultimate line we applied the mean value theorem to obtain $|A_{\tau,j}(x)-T_{x,y}^*A_{\tau,j}(y)T_{x,y}|\leq2\Vert dA_{\tau,j}\Vert_{L^\infty}d(x,y)$. The claim then easily follows from the definition of $H_{x,\tau,j}$.
\begin{align}
    H_{y,\tau,j}f_j(z)&:=\int_{U_{\tau,j}(z)}f_j(w)G_{x,\tau,j}(e_{B_j(x)}^{-1}(z)-e_{B_j(x)}^{-1}(w))dw\nonumber\\
    &\leq(1+\tau^\eta)\int_{U_{\tau,j}(z)}f_j(w)G_{x,\tau,j}(e_{B_j(x)}^{-1}(z)-e_{B_j(x)^{-1}}(w)dw=(1+\tau^\eta) H_{y,\tau,j}f_j(z)
\end{align}
\end{proof}
\begin{lem}[Truncation of Gaussians]\label{trunc}
Let $m,n\in\mathbb{N}$, $\kappa\simeq 1$, and for each $\tau>0$ let $A_{\tau}\in\R^{n\times n}$ be a positive definite matrix. Define $g_{\tau}:\mathbb{R}^n\rightarrow\mathbb{R}$ to be the gaussian $g_{\tau}(x):=\tau^{-n}\exp(-\pi\tau^{-2}\langle A_{\tau}x,x\rangle)$. Suppose that $|\det(A_\tau)|^{-1},|A_{\tau}^{-1}|\leq \tau^{-\epsilon}$ for some $\epsilon\in(0,(1-\gamma)/2)$. Then, there exists a $\nu>0$ depending only on $n$, $m$,  $\epsilon$, and $\gamma$ such that for all $\tau\in(0,\nu)$
\begin{align}
    \det(A_\tau)^{-1/2}=\int_{\mathbb{R}^n} g_{\tau}\leq(1+\tau^{\epsilon})\int_{ U_{\kappa\tau^{\gamma}}(0)}g_{\tau}.\label{eq:trunc}
\end{align}
\end{lem}
The reader should note that this lemma is very similar, albeit slightly stronger, than Lemma 6.5 of \cite{bennett2020nonlinear}, and the proof in turn also follows a similar strategy.
\begin{proof}
Since we may be flexible with the choice we have in $\gamma$, if the claim holds for $\kappa=1$ and we let $\gamma'=\gamma-\eta$ for some small $\eta>0$, we obtain the general result by enlarging the domain of integration on the right-hand side of (\ref{eq:trunc}) from $U_{\tau^{\gamma}}$ to $U_{\kappa\tau^{\gamma'}}(0)$, taking $\tau\leq\kappa^{-1/\eta}$. It is sufficient to show that there exists $\nu>0$ such that, for all $\tau\in(0,\nu)$,
\begin{align}
    \int_{\mathbb{R}^n\setminus U_{\tau^{\gamma}}(0)}g_{\tau}\leq c\tau^{2\epsilon}. \label{eq:trunceq1}
\end{align}
for some $c\simeq 1$. To see this we simply split the integral of $g_\tau$ into $U_{\tau^\gamma}(0)$ and $\R^n\setminus U_{\tau^\gamma}(0)$.
\begin{align*}
    \det(A_\tau)^{-1/2}&=\int_{\R^n\setminus U_{\tau^\gamma}(0)}g_\tau+\int_{U_{\tau^\gamma}(0)}g_\tau \leq c\tau^{2\epsilon}+\int_{U_{\tau^\gamma}(0)}g_\tau
\end{align*}
We then obtain the desired bound by rearranging the above expression and applying the hypothesis that $|\det(A_\tau)|^{-1}\leq \tau^{-\epsilon}$.
\begin{align*}
     \det(A_\tau)^{-1/2}-c\tau^{2\epsilon}&\leq\int_{U_{\tau^\gamma}(0)}g_\tau\\
      \det(A_\tau)^{-1/2}&\leq(1-c\det(A_\tau)^{-1/2}\tau^{2\epsilon})^{-1}\int_{U_{\tau^\gamma}(0)}g_\tau\leq (1-c\tau^{3\epsilon/2})^{-1}\int_{U_{\tau^\gamma}(0)}g_\tau
\end{align*}
 which of course implies (\ref{eq:trunc}) if $\tau$ is taken to be sufficiently small. To estimate the left hand side of (\ref{eq:trunceq1}), we shall partition the domain of integration $\mathbb{R}^n\setminus U_{\tau^\gamma}(0)$ into annuli, and bound the resulting infinite sum above by a lacunary series.
\begin{align*}
    \int_{\mathbb{R}^n\setminus U_{\tau^\gamma}(0)}g_{\tau}&=\int_{|x|\geq\tau^{\gamma-1}}\exp(-\pi|A_{\tau}^{1/2}x|^2)dx\\
    &=\sum_{k=0}^{\infty}\int_{2^k\tau^{\gamma-1}\leq|x|\leq 2^{k+1}\tau^{\gamma-1}}\exp(-\pi|A_{\tau}^{1/2}x|^2)dx\\
    &\leq \sum_{k=0}^{\infty}\sup_{2^k\log(1/\tau)=|x|} (\exp(-\pi|A_{\tau}x|^2))Vol(\lbrace{2^k\tau^{\gamma-1}\leq|x|\leq 2^{k+1}\tau^{\gamma-1}\rbrace})\\
    &\lesssim\tau^{n(1-\gamma)}\sum_{k=0}^{\infty}2^{nk}\exp(-\pi |A_{\tau}^{-1}|^{-2}2^{2k}\tau^{2(\gamma-1)})\lesssim\tau^{n(1-\gamma)}\sum_{k=0}^{\infty}\exp(-\pi |A_{\tau}^{-1}|^{-2}2^{2k-1}\tau^{2(\gamma-1)})
\end{align*}
Finally, we reduce this bound by applying the hypotheses $|A_\tau^{-1}|\leq\tau^{-\epsilon}$, $\epsilon\leq (1-\gamma)/2$ and the fact that $\tau^{\gamma - 1}\geq c \log(1/\tau)$ for some suitable $c\simeq 1$. We then obtain \eqref{eq:trunc} if we take $\tau\in(0,\nu)$, where $\nu\in(0,1)$ satisfies $c\nu^{2\epsilon+\gamma-1}\geq 2\epsilon$.
\begin{align*}
    \int_{\mathbb{R}^n\setminus U_{\tau^\gamma}(0)}g_{\tau}\lesssim \tau^{2n\epsilon} \sum_{k=0}^{\infty}\tau^{c\tau^{2\epsilon}2^{2k-1}\tau^{\gamma-1}}
    \lesssim \tau^{2n\epsilon}\sum_{k=0}^{\infty}\tau^{c\nu^{2\epsilon+\gamma-1}2^{2k-1}}\lesssim \tau^{2n\epsilon}\sum_{k=0}^{\infty}\tau^{\epsilon2^{2k-1}}\lesssim\tau^{2(n+1)\epsilon}
\end{align*}
\end{proof}
We may now prove the pointwise convergence to initial data for $H_{x,\tau,j}f_j\circ B_j$, a fact the reader will recall that we needed to prove that Proposition \ref{timestep} implies Theorem \ref{NBall1}.
\begin{lem}[Pointwise convergence to initial data]\label{pointwconv}
For each $j\in\{1,...,m\}$, let $f_j\in C_0(M_j)$ and $x\in M$, then, 
\begin{align}
    \lim_{\tau\rightarrow 0}H_{x,\tau,j}f_j\circ B_j(x)= f_j\circ B_j(x).
\end{align}
\end{lem}
\begin{proof}[Proof of Lemma \ref{pointwconv}]
     Let $\tau>0$ be small. Since $f_j$ is uniformly continuous, given $\epsilon>0$, there exists a $\delta>0$ such that for all $z,z'\in M$ such that $d(z,z')\leq\delta$, we have that $|f_j(z)-f_j(z')|\leq\epsilon$. Therefore, provided $C\tau^{\gamma}\leq\delta$, we may bound $|f_j\circ B_j(x)-H_{x,\tau,j}f_j\circ B_j(x)|$ in the following way:
\begin{align*}
    |f_j\circ B_j(x)-H_{x,\tau,j}f_j\circ B_j(x)|=&\left|f_j\circ B_j(x)-\int_{U_{\tau^\gamma,j}(0)}f_j(w)G_{x,\tau,j}(e_{B_j(x)}^{-1}(w))dw\right|\\
    \leq& f_j\circ B_j(x)\left(1-\int_{U_{\tau^\gamma,j}(B_j(x))}G_{x,\tau,j}\circ e_{B_j(x)}^{-1}\right)\\
    &+\left|\int_{V_{\tau^\gamma,j}(0)}\left(f_j\circ B_j(x))-f_j(w)\right)G_{x,\tau,j}(e_{B_j(x)}^{-1}(w))dw\right|.
\end{align*}
By the uniform boundedness of the second derivative of the exponential map $e_{B_j(x)}$ established in Lemma \ref{geolem}, provided that $\tau>0$ is sufficiently small, for all $x\in M$, $j\in\{1,...,m\}$, and $v\in V_{\tau^{\gamma},j}(0)\subset T_{B_j(x)} M_j$, we have
\begin{align}
    (1+\tau^\eta)^{-1}\leq \det(de_{B_j(x)}[v])\leq 1+\tau^\eta. \label{eq:dexp}
\end{align}
We may then apply Lemma \ref{geolem} to bound the first term by a power of $\tau$. For the second term, we apply the triangle inequality and bound the resulting gaussian integral similarly.
\begin{align*}
|f_j&\circ B_j(x)-H_{x,\tau,j}f_j\circ B_j(x)|\\
&\leq f_j\circ B_j(x)\left(1-(1+\tau^{\eta})^{-1}\int_{V_{\tau^{\gamma},j}(0)}G_{x,\tau,j}\right)\\
&\hspace{2cm}+\int_{U_{\tau^\gamma,j}(B_j(x))}\left|f_j\circ B_j(x)-f_j(w)\right|G_{x,\tau,j}\circ e_{B_j(x)}^{-1}(w)dw\\
    &\leq (1+\tau^{\eta})^{-1}(f_j\circ B_j(x)\tau^{\eta}+\epsilon)
\end{align*}
This of course implies the claim of the lemma.
\end{proof}
\begin{lem}[Switching]\label{switch}
     For all $\eta\in(0,\alpha)$, there exists $\nu>0$ such that for $\tau\in(0,\nu)$ and $x,y\in M$ such that $d(x,y)\leq\tau^\gamma$,
        \begin{align}
\frac{1}{\BL(\txtbf{dB}(y),\txtbf{p})}\prod_{j=1}^mg_{y,\tau,j}\circ dB_j(y)(e_y^{-1}(x))^{p_j}\leq\frac{1+\tau^{\eta}}{\BL(\txtbf{dB}(x),\txtbf{p})}\prod_{j=1}^mg_{x,\tau,j}\circ dB_j(x)(e_x^{-1}(y))^{p_j}. \label{eq:switch}
\end{align}
\end{lem}
\begin{proof}
     Let $0<\tau<\rho/10$, and define the positive-definite symmetric matrix field $M_\tau\in\Gamma(TM\otimes T^*M)$ by
    \begin{align*}
         M_\tau(x):=\sum_{j=1}^mp_jdB_j(x)^*a_{\tau,j}(x)dB_j(x).
    \end{align*}
    It follows from the definition of $M_\tau$ that
    \begin{align*}
        \prod_{j=1}^mg_{x,\tau,j}\circ dB_j(x)(v)^{p_j}&=\prod_{j=1}^m\exp(p_j\langle a_{\tau,j}(x)dB_j(x)v,dB_j(x)v\rangle)\\
        &=\exp(-\pi\tau^{-2}|M_\tau(x)^{1/2}v|^2).
    \end{align*}
    Hence, by the fact that $\txtbf{a}_{\tau,j}(x)$ is a $\tau^{\alpha}$-near extremiser for $(\txtbf{dB}(x),\txtbf{p})$, 
    \begin{align*}
        (1-\tau^{\alpha})\BL(\txtbf{dB}(x),\txtbf{p})\leq\BLg(\txtbf{dB}(x),\txtbf{p};\txtbf{a}_{\tau,j}(x))=\det(M_\tau(x))^{-1/2}\leq\BL(\txtbf{dB}(x),\txtbf{p}),
    \end{align*}
    so
\begin{align*}
(1-\tau^{\alpha})\det(M_\tau(x))^{1/2}\tau^{-n}\exp\left(-\frac{\pi}{\tau^2}\vert M_\tau(x)^{1/2}v\vert^2\right)&\leq\frac{\prod_{j=1}^mg_{x,\tau,j}\circ dB_j(x)(v)^{p_j}}{\BL(\txtbf{dB}(x),\txtbf{p})}\\
&\leq\det(M_\tau(x))^{1/2}\tau^{-n}\exp\left(-\frac{\pi}{\tau^2}\vert M_\tau(x)^{1/2}v\vert^2\right).
\end{align*}
Taking logarithms of the ratio of the two quantities arising on either side of (\ref{eq:switch}) reveals that the logarithm of the error factor in \eqref{eq:switch} is polynomial in $\tau$.
\begin{align*}
&\log\left(\frac{\BL(\txtbf{dB}(x),\txtbf{p})\prod_{j=1}^mg_{y,\tau,j}\circ dB_j(y)(e_y^{-1}(x))^{p_j}}{\BL(\txtbf{dB}(y),\txtbf{p})\prod_{j=1}^mg_{x,\tau,j}\circ dB_j(x)(e_x^{-1}(y))^{p_j}}\right)\\
&\leq\log\left(\frac{\exp\left(-\pi\tau^{-2}\vert M_\tau(y)^{1/2}e_y^{-1}(x)\vert^2\right)\det(M_\tau(y))^{1/2}}{(1-\tau^{\alpha})\exp\left(-\pi\tau^{-2}\vert M_\tau(x)^{1/2}e_x^{-1}(y)\vert^2\right)\det(M_\tau(x))^{1/2}}\right)\\
&\leq\pi\tau^{-2}\left(|M_\tau(y)^{1/2}e_y^{-1}(x)|^2-|M_\tau(x)^{1/2}e_x^{-1}(y)|^2\right)+\log(\det(M_\tau(y)M_\tau(x)^{-1}))-\log(1-\tau^\alpha)\
\end{align*}
Let $\sigma:I\rightarrow M$ be a geodesic such that $\sigma(0)=x$ and $\sigma(1)=y$. Let $P_\sigma:T_xM\rightarrow T_yM$ denote parallel transport along $\sigma$. It is straightforward to check that $e_y^{-1}(x):=-P_\sigma e_x^{-1}(y)$, hence we may collate the two squares in the first term, allowing us to bound the resulting quantity using the mean value theorem.
\begin{align*}
&\log\left(\frac{\BL(\txtbf{dB}(x),\txtbf{p})\prod_{j=1}^mg_{y,\tau,j}\circ dB_j(y)(e_y^{-1}(x))^{p_j}}{\BL(\txtbf{dB}(y),\txtbf{p})\prod_{j=1}^mg_{x,\tau,j}\circ dB_j(x)(e_x^{-1}(y))^{p_j}}\right)\\
&\leq\pi\tau^{-2}\langle (P_\sigma^{-1} M_\tau(y)P_\sigma-M_\tau(x))e_x^{-1}(y),e_x^{-1}(y)\rangle+\log(\det(M_\tau(y)M_\tau(x)^{-1}))-\log(1-\tau^\alpha)\\
&\lesssim \tau^{\gamma-2}\Vert dM_\tau\Vert_{L^\infty}|e_x^{-1}(y)|^2+\tau^{\gamma-2\epsilon}+\tau^{\alpha}\lesssim\tau^{3\gamma-2-\epsilon}+\tau^{\gamma-2\epsilon}+\tau^{\alpha}\lesssim\tau^{\alpha}.
\end{align*}
Provided $\tau$ is taken to be sufficiently small, we then obtain the desired upper bound.
\begin{align*}
    \frac{\BL(\txtbf{dB}(x),\txtbf{p})\prod_{j=1}^mg_{y,\tau,j}\circ dB_j(y)(e_y^{-1}(x))^{p_j}}{\BL(\txtbf{dB}(y),\txtbf{p})\prod_{j=1}^mg_{x,j,\tau}\circ dB_j(x)(e_x^{-1}(y))^{p_j}}\leq\exp(c\tau^\alpha)\leq 1+\tau^{\eta},
\end{align*}
where $c\simeq 1$.
\end{proof}
The next lemma ensures that the we may perturb the operators $H_{x,\tau,j}$ in $x$ at the expense of a quantitatively small multiplicative error, and it will be a key tool not only for proving our theorem but also for proving Corollaries \ref{cor1} and \ref{cor2}. It shall become clear why it is essential that we truncate the gaussians $G_{x,\tau,j}$, as gaussians are not locally constant at any scale unless restricted to a ball of suitable size with respect to the scale of mollification, and the choice of $\tau^\gamma$ is well suited to our purposes. The reader should note that this lemma plays a similar role to Lemma 6.7 of \cite{bennett2020nonlinear}, although the way that local-constancy is expressed and the scales involved differ slightly between them.\par
In order to perturb to a nearby gaussian the radius of truncation needs to be slightly increased, and so we therefore shall need to define a minor modification of $H_{x,\tau,j}$, where the radius of the domain of integration is multiplied by a factor of 1.1. This factor is of course chosen arbitrarily, but since this consideration is a minor technicality we simply choose a value for the sake of concreteness.
\begin{align*}
    H_{x,\tau,j}^{1.1}&:L^1(M_j)\rightarrow L^1(U_{\rho-1.1\tau^{\gamma}}(B_j(x)))\\
    H_{x,\tau,j}^{1.1}f_j(z)&:=\int_{U_{1.1\tau^{\gamma},j}(z)}f_j(w)G_{x,\tau,j}(e_{B_j(x)}^{-1}(z)-e_{B_j(x)}^{-1}(w))dw
\end{align*}
\begin{lem}[Local-constancy]\label{loconst}
For any $\eta\in(0,\gamma-\epsilon)$, there exists a $\nu>0$ such that the following holds for all $\tau\in(0,\nu)$: Let $x\in M$, then given $z,\tilde{z}\in U_{\tau,j}(B_j(x))$ such that $d(z,\tilde{z})\lesssim\tau^2$ we have that for all $f_j\in L^1(M_j)$,
\begin{align}
    H_{x,\tau,j}f_j(z)\leq (1+\tau^{\eta})H^{1.1}_{x,\tau,j}f_j(\tilde{z})\label{eq:loconst}
\end{align}

\end{lem}
\begin{proof}

First of all we need to prove a similar claim for the kernel $G_{x,\tau,j}$. Suppose that $v,w\in T_{B_j(x)}M_j$ are such that $\vert v-w\vert\leq\kappa\tau^2$ for some $\kappa\simeq 1$ and $v,w\in V_{\tau^\gamma,j}(0)$.
\begin{align*}
\frac{G_{x,\tau,j}(v)}{G_{x,\tau,j}(w)}&=\exp(\pi\tau^{-2}(\vert A_{\tau,j}(x)^{1/2} v\vert^2-\vert A_{\tau,j}(x)^{1/2} w\vert^2))\\
&=\exp(\pi\tau^{-2}\langle A_{\tau,j}(x)(v-w),v+w\rangle)\\
&\leq\exp(\pi\tau^{-2}\Vert A_{\tau,j}\Vert\vert v-w\vert\vert v+w\vert)\\
&\leq\exp(2C^2\kappa\pi\tau^{-2}\tau^{-\epsilon}\tau^2\tau^\gamma)\\
&=\exp(2C^2\kappa\pi\tau^{\gamma-\epsilon})
\end{align*}
Hence it follows that, for all $\tau>0$ sufficiently small, $d(z,\tilde{z})\lesssim\tau^2$, and $w\in U_{\tau,j}(z)$,
\begin{align*}
    G_{x,\tau,j}(e_{B_j(x)}^{-1}(z)-e_{B_j(x)}^{-1}(w))\leq(1+\tau^\eta)G_{x,\tau,j}(e_{B_j(x)}^{-1}(\tilde{z})-e_{B_j(x)}^{-1}(w)).
\end{align*}
The lemma then follows from applying this bound directly to the definition of $H^{1.1}_{x,\tau,j}f_j$.
\begin{align*}
    H_{x,\tau,j}f_j(z)&=\int_{U_{\tau^{\gamma},j}(z)}f_j(w)G_{x,\tau,j}(e_{B_j(x)}^{-1}(z)-e_{B_j(x)}^{-1}(w))dw\\
    &\leq (1+\tau^{\eta})\int_{U_{1.1\tau^{\gamma},j}(z)}f_j(w)G_{x,\tau,j}(e_{B_j(x)}^{-1}(\tilde{z})-e_{B_j(x)}^{-1}(w))dw\\
    &= (1+\tau^{\eta})H_{x,\tau,j}^{1.1}f_j(\tilde{z})
\end{align*}

\end{proof}


\section{Proof of Proposition \ref{timestep}}
Our proof strategy is to use the near-extremising gaussians $g_{x,\tau,j}$ to construct a partition of unity for the integral on the left-hand side of (\ref{eq:Nball1}), subordinate to balls of scale $\tau^\gamma$. At this scale, we may apply our lemmas from the previous section to perturb the integral, so that we may then apply the linear Brascamp--Lieb inequality locally, thereby obtaining the desired form on the right-hand side. Gaussian partitions of unity were also used in \cite{bennett2020nonlinear}, and, notably,  more recently in the context of decoupling for the parabola by Guth, Maldague, and Wang \cite{guth2020improved}.
\begin{proof}
For each $j\in\{1,...,m\}$, take some arbitrary $f_j\in L^1(M_j)$. Let $\eta\in(0,\min\{\alpha,0.9\gamma-\epsilon,\gamma-2\epsilon,3\gamma-2-\epsilon\})$ and choose $\nu>0$ such that \eqref{eq:lem2}, \eqref{eq:trunc}, \eqref{eq:dexp}, \eqref{eq:switch}, and \eqref{eq:loconst} hold for $\tau\in(0,\nu)$. Consider the following collection of truncated gaussians.
\begin{align}
    \left\lbrace\frac{\chi_{V_{0.1\tau^{\gamma}}(0)}\prod_{j=1}^mg_{y,\tau,j}\circ dB_j(y)^{p_j}}{\BL(\txtbf{dB}(y),\txtbf{p})}\right\rbrace_{y\in M}  \label{partofun}
\end{align}
By Lemma \ref{trunc} and the fact that $\txtbf{a}_{\tau}(y)$ is a $\tau^{\alpha}$-near extremiser for $(\txtbf{dB}(y),\txtbf{p})$, we know that for $\tau>0$ sufficiently small,
\begin{align*}
    \BL(\txtbf{dB}(y),\txtbf{p})\leq(1+\tau^\eta)\BLg(\txtbf{dB}(y),\txtbf{p};\txtbf{a}_{\tau}(y))&\leq(1+\tau^\eta)^2\int_{V_{0.1\tau^{\gamma}}(0)}\prod_{j=1}^mg_{y,\tau,j}\circ dB_j(y)(v)^{p_j}
\end{align*}
hence we may continuously split up the integral on the left-hand side of (\ref{eq:Nball1}) by introducing (\ref{partofun}) as one might a partition of unity.
\begin{align*}
    &\int_{U+U_{\tau^\gamma}(0)}\prod_{j=1}^mH_{y,\tau,j}f_j\circ B_j(y)^{p_j}dy\\
    &\leq(1+\tau^\eta)^2 \int_{U+U_{\tau^\gamma}(0)}\int_{V_{0.1\tau^{\gamma}}(y)}\prod_{j=1}^mH_{y,\tau,j}f_j\circ B_j(y)^{p_j}g_{y,\tau,j}\circ dB_j(y)(v)^{p_j}dv\frac{dy}{\BL(\txtbf{dB}(y),\txtbf{p})}\\
     &\leq(1+\tau^\eta)^3 \int_{U+U_{\tau^\gamma}(0)}\int_{ U_{0.1\tau^{\gamma}}(y)}\prod_{j=1}^mH_{y,\tau,j}f_j\circ B_j(y)^{p_j}g_{y,\tau,j}\circ dB_j(y)(e_{y}^{-1}(x))^{p_j}dx\frac{dy}{\BL(\txtbf{dB}(y),\txtbf{p})}\\
    &\leq(1+\tau^\eta)^3 \int_{U+U_{2^{1/2}\tau^\gamma}(0)}\int_{U_{0.1\tau^{\gamma}}(x)}\prod_{j=1}^mH_{y,\tau,j}f_j\circ B_j(y)^{p_j}g_{y,j,\tau}\circ dB_j(y)(e_{y}^{-1}(x))^{p_j}\frac{dy}{\BL(\txtbf{dB}(y),\txtbf{p})}dx
\end{align*}
We want to perturb the inner integral to a linear Brascamp--Lieb inequality in $y$. To do this, we first apply Lemma \ref{switchH} and Lemma \ref{switch} to remove some of the unwanted $y$-dependence. Let $P:=\sum_{j=1}^mp_j$, then
\begin{align*}
  &\int_{U+U_\tau(0)}\prod_{j=1}^mH_{y,\tau,j}f_j\circ B_j(y)^{p_j}dy\\
    &\leq (1+\tau^\eta)^{3+P} \int_{U+U_\tau(0)}\int_{U_{0.1\tau^{\gamma}}(x)}\prod_{j=1}^mH_{x,\tau,j}f_j\circ B_j(y)^{p_j}g_{y,\tau,j}\circ dB[y](e_{y}^{-1}(x))^{p_j}\frac{dy}{\BL(\txtbf{dB}(y),\txtbf{p})}dx\\
    &\leq (1+\tau^\eta)^{3+2P} \int_{U+U_{2^{1/2}\tau^\gamma}(0)}\int_{U_{0.1\tau^{\gamma}}(x)}\prod_{j=1}^mH_{x,\tau,j}f_j\circ B_j(y)^{p_j}g_{x,\tau,j}\circ dB[x](e_{x}^{-1}(y))^{p_j}dy\frac{dx}{\BL(\txtbf{dB}(x),\txtbf{p})}\\
    &\leq (1+\tau^\eta)^{3+3P} \int_{U+U_{2^{1/2}\tau^\gamma}(0)}\int_{V_{0.1\tau^{\gamma}}(x)}\prod_{j=1}^mH_{x,\tau,j}f_j\circ B_j(e_x(v))^{p_j}g_{x,\tau,j}\circ dB[x](v)^{p_j}dv\frac{dx}{\BL(\txtbf{dB}(x),\txtbf{p})}.
\end{align*}
We may then use Lemma \ref{loconst} to replace the instance of $B_j(e_x(v))$ with its affine approximation around $x$, given by $L^x_j(v):=e_{B_j(x)}(dB_j(x)v)$.
\begin{align*}
    &\leq (1+\tau^{\eta})^{3+4P}\int_{U+U_{2^{1/2}\tau^\gamma}(0)}\int_{V_{0.1\tau^{\gamma}}(x)}\prod_{j=1}^mH_{x,\tau,j}^{1.1}f_j\circ L^x_j(v)^{p_j}g_{x,\tau,j}\circ dB_j(x)(v)^{p_j}dv\frac{dx}{\BL(\txtbf{dB}(x),\txtbf{p})}\\
    &\leq (1+\tau^{\eta})^{3+4P}\int_{U+U_{2^{1/2}\tau^\gamma}(0)}\int_{T_xM}\prod_{j=1}^mH_{x,\tau,j}^{1.1}f_j\circ L^x_j(v)^{p_j}g_{x,\tau,j}\chi_{V_{0.1\tau^{\gamma},j}(0)}\circ dB_j(x)(v)^{p_j}dv\frac{dx}{\BL(\txtbf{dB}(x),\txtbf{p})}.
\end{align*}
Above we used the fact that, for all $x\in M$  $V_{0.1\tau^{\gamma}}(0)\subset\bigcap_{j=1}^mdB_j(x)^{-1}V_{0.1\tau^{\gamma},j}(0)$. At this point we may apply the linear Brascamp--Lieb inequality  $(\txtbf{dB}(x),\txtbf{p})$ to the inner integral.
\begin{align}
    &\leq(1+\tau^{\eta})^{3+4P}\int_{U+U_{2^{1/2}\tau^\gamma}}\prod_{j=1}^m\left(\int_{V_{0.1\tau^{\gamma},j}(0)}H^{1.1}_{x,\tau,j}f_j(e_{B_j(x)}(v_j))g_{x,\tau,j}(v_j)dv_j\right)^{p_j}dx\label{eq:penult}
\end{align}
The resulting integrals in (\ref{eq:penult}) may be then be bounded by a convolution.
\begin{align}
    \int_{V_{0.1\tau^{\gamma},j}(0)}&H^{1.1}_{x,\tau,j}f_j(e_{B_j(x)}(v_j))g_{x,\tau,j}(v_j)dv_j\nonumber\\
    &=\int_{V_{0.1\tau^{\gamma},j}(0)}\int_{U_{1.1\tau^{\gamma},j}(B_j(x))}f_j(z)G_{x,\tau,j}(v_j-e_{B_j(x)}^{-1}(z))g_{x,\tau,j}(v_j)dzdv_j\nonumber\\
    &\leq(1+\tau^\eta)\int_{V_{0.1\tau^{\gamma},j}(0)}\int_{V_{1.1\tau^{\gamma},j}(B_j(x))}f_j\circ e_{B_j(x)}(w)G_{x,\tau,j}(v_j-w)g_{x,\tau,j}(v_j)dwdv_j\nonumber\\
    &=G_{x,\tau,j}\chi_{V_{1.1\tau^{\gamma}}(0)}\ast g_{x,\tau,j}\chi_{V_{0.1\tau^{\gamma}}(0)}\ast f_j\circ e_{B_j(x)}(0)\label{eq:conv}
\end{align}
Now, $G_{x,\tau,j}\ast g_{x,\tau,j}= G_{x,2^{1/2}\tau,j}$ by definition of $G_{x,\tau,j}$, and the support of $\chi_{V_{1.1\tau^{\gamma}}(0)}\ast\chi_{V_{0.1\tau^{\gamma}}(0)}$ is the ball around the origin of radius $1.2\tau^{\gamma}$, which is less than $2^{\gamma/2}\tau^{\gamma}$ provided that $\gamma\geq 2\log_2(1.2)\approx 0.526...$. This implies that $\supp(G_{x,\tau,j}\chi_{V_{1.1\tau^{\gamma}}(0)}\ast g_{x,\tau,j}\chi_{V_{0.1\tau^{\gamma}}(0)})\subset V_{2^{\gamma/2}\tau^{\gamma}}(0)$, hence 
\begin{align*}
    G_{x,\tau,j}\chi_{V_{1.1\tau^{\gamma}}(0)}\ast g_{x,\tau,j}\chi_{V_{0.1\tau^{\gamma}}(0)}\leq(G_{x,\tau,j}\ast g_{x,\tau,j})\chi_{V_{2^{\gamma}\tau^{\gamma}}(0)}=G_{x,2^{1/2}\tau,j}\chi_{V_{2^{\gamma/2}\tau^{\gamma}}(0)}
\end{align*}
We may then bound (\ref{eq:conv}) as follows:
\begin{align}
    \int_{V_{0.1\tau^{\gamma},j}(0)}H^{1.1}_{x,\tau,j}f_j(e_{B_j(x)}(v_j))g_{x,\tau,j}(v_j)dv_j\leq H_{x,2^{1/2}\tau,j}f_j\circ B_j(x).\label{eq:fin}
\end{align}
Finally, we complete the proof by combining (\ref{eq:penult}) with (\ref{eq:fin}) and taking $\beta\in(0,\eta)$.
\end{proof}
\section{Proof of Corollaries \ref{cor1} and \ref{cor2}}
\begin{proof}[Proof of Corollary \ref{cor1}.]
Take some arbitrary $f_j\in L^1(M_j)$ for all $j\in\{1,...,m\}$. By Theorem \ref{NBall1}, there exists a $\beta>0$ such that for $\tau>0$ sufficiently small
\begin{align*}
    \int_{U_{\tau^\gamma}(x_0)}\prod_{j=1}^mf_j\circ B_j(x)^{p_j}dx
    &\leq (1+\tau^{\beta})\int_{U_{2\tau^\gamma}(x_0)}\prod_{j=1}^m H_{x,\tau,j}f_j\circ B_j(x)^{p_j}dx
\end{align*}
Take $\eta$ and $\nu$ as in the proof of Proposition \ref{timestep}, if we take $\tau\in(0,\nu)$, then we may apply Lemma \ref{switchH} to perturb $H_{x,\tau,j}$ to $H_{x_0,\tau,j}$ and Lemma \ref{loconst} to perturb $B_j(x)$ to $L_j^{x_0}(x)$, at which point we may apply the linear inequality to complete the proof.
\begin{align*}
    &\leq(1+\tau^{\beta})(1+\tau^{\eta})^P\int_{U_{2\tau^\gamma}(x_0)}\prod_{j=1}^m H_{x_0,\tau,j}f_j\circ B_j(x)^{p_j}dx\\
    &\leq(1+\tau^{\beta})(1+\tau^{\eta})^{2P}\int_{U_{2\tau^\gamma}(x_0)}\prod_{j=1}^m H_{x_0,\tau,j}^{1.1}f_j\circ L^{x_0}_j(x)^{p_j}dx\\
    &\leq(1+\tau^{\beta})(1+\tau^{\eta})^{2P}\BL(\txtbf{dB}(x_0),\txtbf{p})\prod_{j=1}^m\left(\int_{U_{2\tau^{\gamma},j}(0)}H_{x_0,\tau,j}^{1.1}f_j\circ e_{B_j(x)}\right)^{p_j}\\
    &\leq(1+\tau^{\beta})(1+\tau^{\eta})^{3P}\BL(\txtbf{dB}(x_0),\txtbf{p})\prod_{j=1}^m\left(\int_{M_j}f_j\right)^{p_j},
\end{align*}
where $P:=\sum_{j=1}p_j$, as in the previous section.
\end{proof}
\begin{proof}[Proof of Corollary \ref{cor2}.]

Fix some $\tau>0$ small enough so that (\ref{eq:Nball1}) holds for the nonlinear datum $(\txtbf{B},\txtbf{p})$. For all $x\in M$, since $d_{M_j}(B_j(x),\widetilde{B}_j(x))\leq\rho$, $e_{B_j(x)}^{-1}(\widetilde{B}_j(x))$ is well-defined. We may then consider the following ratio for all $v\in V_{\tau^\gamma,j}(0)\subset T_{B_j(x)}M_j$.
\begin{align}
    \frac{G_{x,\tau,j}(v)}{G_{x,\tau,j}(e_{B_j(x)}^{-1}(\widetilde{B}_j(x))-v)}&=\exp\left(\pi\tau^{-2}(|A_{x,\tau,j}^{1/2}(e_{B_j(x)}^{-1}(\widetilde{B}_j(x))-v)|^2-|A_{x,\tau,j}^{1/2}v|^2)\right)\nonumber\\
    &=\exp\left(\pi\tau^{-2}\langle A_{x,\tau,j}(e_{B_j(x)}^{-1}(\widetilde{B}_j(x))-2v),e_{B_j(x)}^{-1}(\widetilde{B}_j(x))\rangle \right)\nonumber\\
    &\leq\exp\left(\pi\tau^{-2}\Vert A_{x,\tau,j}\Vert|e_{B_j(x)}^{-1}(\widetilde{B}_j(x))|(|e_{B_j(x)}^{-1}(\widetilde{B}_j(x))|+2|v|)\right)\nonumber
\end{align}
Because $|e_{B_j(x)}^{-1}(\widetilde{B}_j(x))|= d_{M_j}(B_j(x),\widetilde{B}_j(x))\leq \rho$, this then implies that
\begin{align}
    \frac{G_{x,\tau,j}(v)}{G_{x,\tau,j}(e_{B_j(x)}^{-1}(\widetilde{B}_j(x))-v)}\lesssim_{\rho,\tau}1.\label{eq:cor2pf}
\end{align}
Define the following convolution operator:
\begin{align}
    H_{\tau,j}f_j(y):=\tau^{-n_j}\int_{U_{\tau^\gamma,j}(y)}f_j(z)\exp(-\pi\tau^{\epsilon-2}|e_{B_j(x)}^{-1}(z)|^2)dz
\end{align}
By Lemma \ref{abounds}, $|A_{x,\tau,j}z|\gtrsim\tau^{\epsilon}|z|$ for all $z\in T_{B_j(x)}M_j$ and all $x\in M$. Combining this with \eqref{eq:cor2pf}, we obtain the bound $H_{x,\tau,j}f_j\circ B_j(x)\lesssim_{\rho,\tau}H_{\tau,j}f_j\circ\widetilde{B}_j(x)$.
\begin{align*}
    H_{x,\tau,j}f_j\circ B_j(x)&=\tau^{-n_j}\int_{U_{\tau^\gamma,j}(B_j(x))}f_j(z)G_{x,\tau,j}(e_{B_j(x)}^{-1}(z))dz\\
    &\lesssim_{\rho,\tau}\int_{U_{\tau^\gamma,j}(B_j(x))}f_j(z)G_{x,\tau,j}(\widetilde{B}_j(x)-e_{B_j(x)}^{-1}(z))dz\\
    &\lesssim_{\tau,\epsilon}\int_{U_{\tau^\gamma,j}(y)}f_j(z)\exp(-\pi\tau^{\epsilon-2}|\widetilde{B}_j(x)-e_{B_j(x)}^{-1}(z)|^2)dz\\
    &=H_{\tau,j}f_j\circ\widetilde{B}_j(x)
\end{align*}
The finiteness of constant associated with $(\txtbf{B},\txtbf{p})$ then follows quickly from \eqref{eq:Nball1} and the finiteness of the constant associated with $(\widetilde{\txtbf{B}},\txtbf{p})$.
\begin{align*}
  \int_M\prod_{j=1}^mf_j\circ B_j(x)^{p_j}dx&\leq (1+\tau^{\beta})\int_M\prod_{j=1}^m H_{x,\tau,j}f_j\circ B_j(x)^{p_j}dx\\
 &\lesssim_{\rho,\tau,\epsilon}\int_M\prod_{j=1}^m H_{\tau,j}f_j\circ \widetilde{B}_j(x)^{p_j}dx\\
  &\lesssim_{\widetilde{\txtbf{B}}}\prod_{j=1}^m\left(\int_{M_j}H_{\tau,j}f_j\right)^{p_j}\\
  &\lesssim_{\tau,\epsilon}\prod_{j=1}^m\left(\int_{M_j}f_j\right)^{p_j}
\end{align*}
\end{proof}
\section{Appendix: a Geometric Lemma}
Here we establish Lemma \ref{geolem}, which asserts that our uniform boundedness assumptions from Section \ref{setup section} imply good uniform control of the first and second order derivatives of the exponential map. The proof uses some standard ideas from the analysis of ODEs, and while the result itself is not new, its proof is included here for the sake of completeness.
\begin{proof}
We should first clarify that, in this proof, double bars shall always denote $L^\infty$ norms. We first prove the case for derivatives of order $1$. Let $p\in M$ and $X,Y\in T_p M$, with $|X|,|Y|< \rho$. We consider the following vector field $J(t):(0,\infty)\rightarrow TM$ defined over the curve parametrised by $\gamma(t):=\exp(tX)$:
$$J(t):=\partial_s\exp_p(t(X+sY))\vert_{s=0}.$$
By definition of the exponential map, $J$ is a Jacobi field with initial data $J(0):=0$ and $J'(0)=Y$, hence it satisfies the Jacobi equation:
\begin{align}
    J''+R(J,\gamma')\gamma'=0 \label{eq: jacobi}
\end{align}
Here $R$ denotes the Riemannian curvature endomorphism. Now, define the following quantity $F(t):=|J(t)|^2+|J'(t)|^2$. We shall aim to bound this quantity via bounding its derivative  using \eqref{eq: jacobi} and the AM-GM inequality.
\begin{align*}
    F'&=2\langle J,J'\rangle + 2\langle J',J''\rangle\\
    &=2(\langle J,J'\rangle + 2\langle J',R(J,\gamma')\gamma'\rangle)\\
    &\leq 2(|J||J'|+|J'|\Vert R\Vert |J||X|^2)\\
    &\leq (1+\Vert R\Vert\rho^2)F
\end{align*}
Hence $F(t)\leq e^{t(1+\Vert R\Vert\rho^2)}F(0)$, and so $$|d\exp_p(X)Y|=J(1)\leq F(1)^{1/2}\leq e^{(1+\Vert R\Vert\rho^2)/2}F(0)^{1/2}=e^{(1+\Vert R\Vert\rho^2)/2}|Y|.$$
We then bootstrap to the second order case via a similar method. Let $Z\in T_pM$, $|Z|<\rho$, and consider the following family of variations of $J$:
$$J_\epsilon(t):=\partial_s\exp_p(t(X+sY+\epsilon Z))\vert_{s=0}$$
Each such $J_\epsilon$ is a Jacobi field for all $\epsilon>0$, so we may then differentiate \eqref{eq: jacobi} in $\epsilon$ to find that
\begin{align*}
    \partial_\epsilon J_\epsilon''+\partial_JR(J_\epsilon(t),\gamma')(\gamma',\partial_\epsilon J_\epsilon)=0, \forall t,\epsilon>0,
\end{align*}
where $\partial_J R$ refers to the partial covariant derivative of the Riemannian curvature tensor in the first argument. We now consider the quantity $G(t):=|\partial_\epsilon J_0(t)|^2+|\partial_\epsilon J_0'(t)|^2$, and apply a similar argument to last time
\begin{align*}
    G'&=2\langle \partial_\epsilon J_0,\partial_\epsilon J_0'\rangle + 2\langle \partial_\epsilon J_0',\partial_\epsilon J_0''\rangle\\
    &=2(\langle \partial_\epsilon J_0,\partial_\epsilon J_0'\rangle + 2\langle J_0',\partial_JR(J_0,\gamma')(\gamma',\partial_\epsilon J_0)\rangle\\
    &\leq 2(|\partial_\epsilon J_0||\partial_\epsilon J_0'|+|\partial_\epsilon J_0'|\Vert \partial_J R\Vert |J_0||\partial_\epsilon J_0||X|^2)\\
    &\leq (1+\Vert \partial_\epsilon R\Vert |J_0|^2\rho^2)G
\end{align*}
Hence $G(t)\leq e^{t(1+\Vert \partial_J R\Vert \rho^2(\sup_{0<l<t}|J_0|^2(l)))}G(0)\leq e^{t(1+\Vert \partial_J R\Vert \rho^2e^{t(1+\Vert R\Vert\rho^2)}|Y|)}G(0)$, therefore,
\begin{align*}
    |d^2\exp(X)(Y,Z)|=\partial_\epsilon J_0(1)\leq G(1)^{1/2}\leq e^{t(1+\Vert \partial_J R\Vert\rho^2 e^{t(1+\Vert R\Vert \rho^2)})/2}G(0)^{1/2}=e^{(1+\Vert \partial_J R\Vert \rho^2e^{(1+\Vert R\Vert\rho^2)/2})/2}|Z|
\end{align*}
By symmetry, we also have that $|d^2\exp(X)(Y,Z)|\leq e^{(1+\Vert \partial_J  R\Vert e^{t(1+\Vert R\Vert)/2})/2}|Y|$, so we are done.
\end{proof}
\bibliographystyle{plain}
\bibliography{Bib}
\end{document}

%% file: preamble.tex
\usepackage{setspace}
\usepackage[margin=3cm]{geometry}
\usepackage{graphicx}
\usepackage{hyperref}
\usepackage{cite}
\usepackage{amsmath}
\usepackage{amssymb}
\usepackage{amsthm}
\usepackage{mathabx}
\usepackage{dsfont}
\usepackage{xcolor}
\usepackage{relsize}

\newtheorem{thm}{Theorem}[section]
\newtheorem{cor}[thm]{Corollary}
\newtheorem{lem}[thm]{Lemma}
\newtheorem{prop}[thm]{Proposition}

\newtheorem{dfn}[thm]{Definition}
\theoremstyle{definition}

\newcommand{\R}{\mathbb{R}}
\newcommand{\N}{\mathbb{N}}
\newcommand{\Z}{\mathbb{Z}}

\renewcommand{\Gamma}{\varGamma}
\renewcommand{\epsilon}{\varepsilon}

\renewcommand{\leq}{\leqslant}
\renewcommand{\geq}{\geqslant}

\newcommand{\txtbf}[1]{\textnormal{\textbf{#1}}}
\newcommand{\BL}{\textnormal{BL}}
\newcommand{\BLg}{\textnormal{BL}_{\textnormal{g}}}
\newcommand{\supp}{\textnormal{supp}}
\newcommand{\adj}{\textnormal{adj}}
\newcommand{\dist}{\textnormal{dist}}